\documentclass[11pt]{article}
\usepackage{amsmath,amsthm,amsfonts,amssymb,mathrsfs}
\usepackage{epsfig}
\usepackage[usenames]{color}
\usepackage{verbatim}
\usepackage{hyperref}
\usepackage{multicol}
\usepackage{comment}
\usepackage{float}
\usepackage[normalem]{ulem}
\usepackage{graphicx}

\usepackage[color=green, textsize=tiny]{todonotes}

\parskip \medskipamount
\numberwithin{equation}{section}

\newtheorem{dfn}{Definition}[section]
\newtheorem{thm}[dfn]{Theorem}
\newtheorem{lem}[dfn]{Lemma}
\newtheorem{cor}[dfn]{Corollary}
\newtheorem{rem}[dfn]{Remark}

\newtheorem{prop}[dfn]{Proposition}
\makeatletter

\title{Non triviality of the percolation threshold and Gumbel fluctuations for   Branching Interlacements}

\author{Bruno Schapira\thanks{Aix-Marseille Universit\'e, CNRS, I2M, UMR 7373, 13453 Marseille; and Universit\'e Claude Bernard Lyon 1, Institut Camille Jordan, UMR 5208. schapira@math.univ-lyon1.fr.}}

\begin{document}

\date{}
\maketitle

\begin{abstract}
We consider the model of Branching Interlacements, introduced by Zhu, which is a natural analogue of Sznitman's Random Interlacements model, where the random walk trajectories are replaced by ranges of some suitable tree-indexed random walks. We first prove a basic decorrelation inequality for events depending on the state of the field on distinct boxes. We then show that in all relevant dimensions, the vacant set  undergoes a nontrivial phase transition regarding the existence of an infinite connected component. Finally we obtain the Gumbel fluctuations for the cover level of finite sets, which is analogous to Belius' result in the setting of Random Interlacements. 
\newline
MSC 2020 \emph{subject classifications.} 60F05, 60K35.
\end{abstract}

\section{Introduction} 
Since its introduction by Sznitman in his seminal paper~\cite{Sz10}, Random Interlacements have proven to be not only an interesting object of study in its own right, see e.g.~\cite{DC1,DC2,DC3,GPF,GPF2,Li20,LiS, NSz,PTex,Pre,Sz23,Sz23bis} for some of the most recent and spectacular progresses on this topic, but also a fundamental tool in the understanding of other models, among which random walks, especially on a torus, play of course a prominent role, see e.g.~\cite{Bel13, Bou,CTex,Li17,Pre, Rod19, Sz17,TexW}, but also loop soup percolation and the Gaussian free field~\cite{DPR, Lupu, Sz12bis, Sz16}.

In this paper we shall be interested in another model called Branching Interlacements, recently introduced by Zhu~\cite{Zhu18} (see also~\cite{ARZ}), which is defined similarly as Random Interlacements, 
except that random walks trajectories are now replaced by ranges of critical Branching random walks conditioned in a certain sense to have an infinite genealogical tree. The latter turns out to be the so called {\it infinite invariant tree} (more precisely its doubly infinite version), previously introduced by Le Gall and Lin~\cite{LGL16}, which is a rooted plane labelled random tree enjoying a wonderful property of invariance by shift on the labels and rerooting. While initially it was introduced with the purpose of studying the range of critical branching random walks, it has later found to be at the heart of a whole potential theory for branching random walks developed by Zhu~\cite{Zhu16,Zhu16b,Zhu19,Zhu21}, also later continued in~\cite{ASS23}.

As in the usual setting of Random Interlacements, the model of Branching Interlacements depends on an intensity parameter $u>0$, that governs the occupation density of the field, denoted $\mathcal I^u$, whose law is characterized by the relation
\begin{equation}\label{def.Iu1}
\mathbb P(\mathcal I^u\cap K =\emptyset) = \exp\big(-u\cdot \textrm{BCap}(K)\big), 
\end{equation}
for any finite set $K\subset \mathbb Z^d$, with $d\ge 5$, where $\textrm{BCap}(K)$, the {\it branching capacity} of $K$, is a functional, introduced in~\cite{Zhu16}, which is the natural analogue of the Newtonian capacity in the setting of critical branching random walks, see Section~\ref{sec.treeRW} for details. The fact that we take the dimension larger than or equal to five here has to do with the fact that it corresponds exactly to the situation where our tree-indexed random walks are transient, as shown in~\cite{BC,LGL16,LSS}, and is thus the natural counterpart of the restriction $d\ge 3$ in the setting of Random Interlacements.

Now, while it can be seen that the set $\mathcal I^u$ always forms a connected set, irrespective of the value of the intensity parameter, see Corollary~\ref{prop.connect} below, by far a more interesting and richer phenomenology governs the percolative behavior of its  complement, the so-called {\it vacant set} $\mathcal V^u = \mathbb Z^d\setminus \mathcal I^u$. In particular in the usual setting of Random Interlacements, some of the first efforts in the study of this model were put in  showing that it undergoes a nontrivial phase transition regarding the existence of an infinite connected component. Indeed, this has been first obtained in sufficiently high dimension in~\cite{Sz12}, before it could be extended to all relevant dimensions, in that case three and more, in~\cite{SSz}. One of the main goal of this paper will be to prove an analogous statement for Branching Interlacements, and thereby confirm R\'ath's prediction from~\cite{R}, about ten years ago, that his simple proof of the nontriviality of the phase transition could be useful in this context as well.

We will come back to this question in more details in a moment, but let us for now investigate some more basic features of our model. Indeed, one very important aspect to keep in mind concerning the sets $\mathcal I^u$ and $\mathcal V^u$, especially when compared to Bernoulli percolation, is that they are strongly correlated fields, in particular correlations decay at a polynomial speed. To state this result in more details, we first need to specify our hypotheses  concerning the jump and offspring distributions of our branching random walks: so we assume in the whole paper that the offspring distribution  is a probability measure $\mu$ on the set of integers $\mathbb N$, with mean one, and a positive and finite variance, and for simplicity we take the jump distribution to be the uniform measure on the neighbors of the origin, see  Section~\ref{sec.treeRW} for more precise definitions. 
We also let $\|\cdot \|$ be the Euclidean norm.  
\begin{prop}\label{prop.cov} 
There exists a constant $c>0$, such that for any $u>0$,  
\begin{equation}
 {\rm Cov} (\mathbf 1\{x\in \mathcal V^u\}, \mathbf 1\{y\in \mathcal V^u\}) \sim \frac{cu}{\|x-y\|^{d-4}}\cdot e^{-2u\cdot \textrm{BCap}(\{0\})}, \qquad \text{as }\|x-y\|\to \infty. 
 \end{equation}
\end{prop}
This result was stated without proof in~\cite{Zhu18}; so we shall provide one here for completeness. However, while interesting in itself, Proposition~\ref{prop.cov} is often not sufficient in practice to describe the fine properties of the model. A much more efficient result is provided by our next theorem. Given a finite and nonempty set $K\subset \mathbb Z^d$, we let ${\rm diam}(K) =1+ \max\{\|x-y\| : x,y\in K\}$, and given two sets $K_1, K_2\subset \mathbb Z^d$, we let ${\rm dist}(K_1,K_2) = \inf\{ \| x-y\| : x\in K_1, y\in K_2\}$ be the distance between $K_1$ and $K_2$. 

\begin{thm}\label{thm.cov}
There exist positive constants $c$ and $C$, such that for any $u\ge 0$, any finite and nonempty $K_1,K_2\subset \mathbb Z^d$,  
and any events $E$ and $F$ depending only on $\mathcal I^u \cap K_1$ and $\mathcal I^u \cap K_2$ respectively,  
\begin{align*}
| {\rm Cov} (E,F) | \le Cu\cdot    \left\{   \frac{{\rm BCap}(K_1) \cdot {\rm BCap}(K_2)}{{\rm dist}(K_1,K_2)^{d-4}} 
+   {\rm BCap}(K_2) \cdot e^{-c\frac{{\rm dist}(K_1,K_2)}{{\rm diam}(K_1)} }
\right\}. 
\end{align*}
\end{thm}
In the usual setting of Random Interlacements, an analogous inequality was proved in~\cite{Sz10} without the second term in the upper bound. However, in our applications, this error term will appear to be completely innocuous. Note that it is not symmetric in $K_1$ and $K_2$, but we are free to reverse their roles and take the smallest of the two possible choices. 
We will now investigate some applications of this result in two directions. On one hand it will allow us to answer the most basic questions regarding the percolative properties of the vacant set, as already mentioned above, and furthermore, it will also enable us to derive the Gumbel fluctuations for the cover level of finite sets.

To describe these results in more details, one needs an additional tool, regarding translation invariance and ergodicity of the fields. We define the map $t_x$ on $\mathbb Z^d$ as the translation by $x\in \mathbb Z^d$. A first application of Theorem~\ref{thm.cov} is the following result.  
\begin{prop}\label{cor.ergodic}
The law of the random set $\mathcal I^u$ is invariant by all the translations $t_x$, for $x\in \mathbb Z^d$, and these maps are ergodic. 
\end{prop}
The proof is identical to the proof of the corresponding result for Random Interlacements, see~\cite[Theorem 2.1]{Sz10}, and will not be reproduced here.  
This result has some important applications. 
As in~\cite{Sz10}, we denote by $\widetilde{\mathcal I}^u$ the subgraph of $\mathbb Z^d$, whose vertex set is $\mathcal I^u$, and whose edge set is the set of edges crossed by one of the branching random walks generating $\mathcal I^u$, see Section~\ref{sec.BI}. Similarly we let $\widetilde{\mathcal V}^u$ be the graph with vertex set $\mathcal V^u$, and edges between any pair of neighboring vertices. 

\begin{cor}\label{prop.connect}
One has for any $u>0$,
\begin{enumerate}
    \item $p(u)=\mathbb P(\widetilde{\mathcal V}^u \text{ admits an infinite connected component})\in \{0,1\}$. 
\item The graph $\widetilde{\mathcal I}^u$ is almost surely connected. 
\item Letting $N(u)$ be the number of connected components of $\widetilde{\mathcal V}^u$, almost surely $N(u)\in \{0,1\}$. 
\end{enumerate}
\end{cor}
Note that finer connectivity properties of the graph $\widetilde{\mathcal I}^u$ were proved in~\cite{PZ}, when $\mu$ is a Geometric distribution. 
In the setting of random interlacements, an analogue of the two first items of  Corollary~\ref{prop.connect} were proved in the seminal paper~\cite{Sz10},
while the third one is due to Teixeira~\cite{Tex}. As for Proposition~\ref{cor.ergodic}, the proof in our setting can be easily adapted from theirs, and as such will be omitted. 

From Corollary~\ref{prop.connect}, and basic monotonicity properties, we know that as the parameter $u$ increases, the vacant set undergoes a phase transition, going from a regime with an (almost surely unique) infinite connected component, to a regime with almost surely no infinite connected component. The transition occurs at the value 
$$u_*=\inf\{u:p(u)=0\} \in [0,\infty]. $$
Our next result addresses the question of knowing whether or not this value is nontrivial, i.e. different from $0$ and infinity\footnote{The answer to this question has been also announced in~\cite{Zhu18}, but to our knowledge the proof never appeared so far.}. Classically, we say that $\mu$ has a finite exponential moment, if there exists $c>0$, such that $\sum_{k\ge 0} e^{ck}\mu(k)<\infty$. 
\begin{thm}\label{thm:pt} Assume that $\mu$ has a finite exponential moment. Then for any $d\ge 5$, 
$$0<u_*<\infty. $$   
\end{thm}
Most of this paper will be dedicated to proving this result. The main issue concerns the lower bound (especially in low dimension); indeed the upper bound can be shown by a straightforward adaptation of the arguments from~\cite{Sz12} or~\cite{R}. While trying to adapt also the proof of~\cite{R} for the lower bound, one is left with showing that, given some sparse set of two-dimensional spheres, the number of those visited by a single random walk indexed by the infinite invariant tree, has sufficiently high exponential moments. In the usual setting of Random Interlacements, this could be deduced relatively easily from the Markov property of random walks, but the problem here is that our tree indexed random walks do not enjoy this fundamental property. To overcome this issue we use some intricate  induction argument, based on the fact that the law of a critical branching random walk  conditioned on hitting a given set can be described precisely in terms of a path of first visit and a set of so-called adjoint trees attached both to the right and to the left of all its vertices, see  Proposition~\ref{Zhu.hit2} and the remark following it. The difficulty then is two-fold. First we need to decorrelate the trees on the right and on the left of this path, and once this is done, we need to deal with a problem purely about simple random walks, namely one needs to control some complicated functional of its full trajectory, conditionally on its final position. As before we manage to decorrelate both by performing some delicate surgery on the trajectory of the walk.

We now come to our second main application of  Theorem~\ref{thm.cov}, which concerns the cover level of finite sets. Given a finite $K\subset \mathbb Z^d$, we let 
$$M(K) = \inf\{u \ge 0 : K\subset \mathcal I^u\}. $$
The next theorem is a complete analogue of  Belius' result in the setting of Random Interlacements~\cite{Bel12}. 

\begin{thm}\label{thm.cover}
There exist constants $c,c'$, such that for any finite nonempty $K\subset \mathbb Z^d$, 
$$\sup_{z\in \mathbb R} \Big|\mathbb P\Big(\rm{BCap}(\{0\})\cdot M(K) -\log |K|\le z\Big) -  \exp(-e^{-z})\Big| \le c|K|^{-c'}. $$ 
\end{thm}
We note that a similar result is expected to hold as well for the cover time of a torus $(\mathbb Z/N\mathbb Z)^d$ by a critical branching random walk, as is the case for usual random walks~\cite{Bel13}. However, at the moment, only the concentration has been obtained by Zhu~\cite{Zhu18}, while the study of fluctuations still appears as a challenging problem.

The paper is organized as follows. In the next section, we give all relevant definitions, recall some important results of Zhu on hitting probabilities, and derive an original and basic monotonicity result for the Branching capacity, Lemma~\ref{strict.bcap}. 
Next, we prove Proposition~\ref{prop.cov} in Section~\ref{sec:prop.cov} and Theorem~\ref{thm.cov} in Section~\ref{sec:thm.cov}. The proof of our main result, Theorem~\ref{thm:pt}, is given in Section~\ref{sec:thm:pt}. Finally we prove Theorem~\ref{thm.cover} in Section~\ref{sec:thm.cover}. 

\section{Preliminaries}
\subsection{Some notation}
Given two nonnegative functions $f$ and $h$, we write $f\lesssim h$ if there exists a constant $C>0$, such that $f(x)\le Ch(x)$, for all $x$,
and $f\asymp h$ if both $f\lesssim h$ and $h\lesssim f$. We write $f\sim h$, if $h$ is positive and $f(x)/h(x)\to 1$, as $x\to \infty$. 

Given $r>0$, and $x\in \mathbb Z^d$, we let $B(x,r)= \{y\in \mathbb Z^d : \|y-x\|\le r\}$, and for $U\subset \mathbb Z^d$, we let 
$\partial U = \{y\in U : \exists z \in \mathbb Z^d\setminus U \text{ with } \|z-y\|=1\}$ denote its 
inner boundary.

We let $\mathbb P_x$ denote the law of a simple random walk starting from $x$, and sometimes drop $x$ from the notation when it is the origin. Given a finite set $K\subset \mathbb Z^d$, we let $H_K$ be the first hitting time of $K$ by a simple random walk, and $H_K^+$ its first return time to $K$. 
Then we consider 
$$\widetilde e_K(x) = \mathbb P_x(H_K^+ = \infty),$$ 
and recall that the Newtonian capacity of $K$ is by definition $\textrm{Cap}(K) = \sum_{x\in K} \widetilde e_K(x)$, see e.g.~\cite{LL} for details. In particular we shall use that $\textrm{Cap}(B(0,r))\lesssim R^{d-2}$.

Denote by $\mathcal T_c$ a  $\mu$-Bienaym\'e-Galton-Watson tree, and for $n\ge 0$, denote by $Z_n$ its number of vertices at generation $n$. We recall Kolmogorov's estimate, see~\cite[Theorem 1 p.19]{AN}, 
\begin{equation}\label{Kolmogorov}
\mathbb P(Z_n\neq 0) \sim \frac{2}{\sigma^2n}, \quad \text{as }n\to \infty,
\end{equation}
where $\sigma^2 = \sum_{k\ge 1}k(k-1)\mu(k)$ denotes the variance of $\mu$. 

\subsection{The infinite invariant tree}
The $\mu$-infinite invariant tree $\mathcal T$ (or infinite invariant tree for short) is a plane labelled tree defined as follows. First the number of offspring of the root, denoted  $\varnothing$, is distributed according to $\tilde \mu(i)= \mu(i-1)$, for all $i\ge 1$. Its first child is a special vertex, while the others are normal. Normal vertices have a number of offspring sampled independently according to $\mu$ and only give birth to normal vertices. Special vertices have a number of offspring sampled according to $\mu_{\textrm{sb}}(i) = i\mu(i)$, for $i\ge 0$. One of its children chosen uniformly at random is declared special, while others are normal. The set of special vertices, together with the root, form an infinite line of vertices called the spine. We label vertices on the right of the spine according to depth-first search order, starting from the root, which has label $0$. On the other hand we label vertices on the left of the spine, including those on the spine,  according to depth-first search from infinity. The subtree formed by vertices with negative labels is called the past tree, and denoted $\mathcal T_-$, while vertices with nonnegative labels is a forest, denoted $\mathcal T_+$. The forest $\mathcal T_+$ enjoys a fundamental property of invariance in law by shift of the labels and rerooting, which was first identified by Le Gall and Lin in~\cite{LGL16} and then extended to the full tree $\mathcal T$  independently  in~\cite{BW22} and~\cite{Zhu18}. More precisely the map which to each vertex with label $n$ assigns the new label $n+1$, and reroot the tree at the vertex with new label $0$ (so formerly the one which had label $-1$), leaves the tree $\mathcal T$ invariant in law, and the same holds for its inverse map, see Figure~\ref{fig.tree}.

\begin{figure}[ht!]\label{fig.tree}
	\begin{center}
		\includegraphics[width=0.31\textwidth]{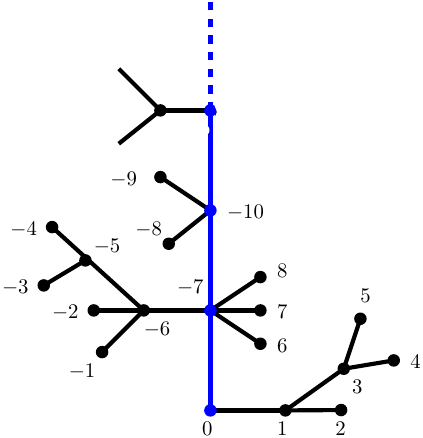}\hspace{0.02\textwidth}
		\includegraphics[width=0.31\textwidth]{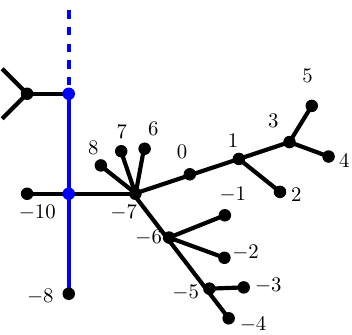}\hspace{0.02\textwidth}
	\includegraphics[width=0.31\textwidth]{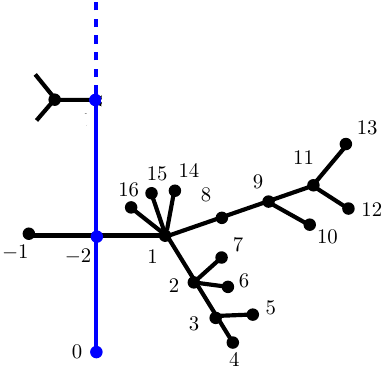}
	\end{center}
	\caption{An infinite tree rooted at $0$, seen from $-8$ and relabelled.}
\end{figure}
An important property of the tree $\mathcal T$ is that the number of offspring in the past of any special vertex has law $\widetilde \mu$ given by $\widetilde \mu(i) = \sum_{k\ge i+1} \mu(k)$, for $i\ge 0$, and similarly for its number of offspring in the future (but note that both numbers are not independent in general, except in the very special case when $\mu$ is distributed as a Geometric random variable). An {\it adjoint tree} is a tree where the root has offspring distribution $\widetilde \mu$, and all other vertices have offspring distribution $\mu$. In particular the trees attached to the special vertices of $\mathcal T$, either in the past or in the future, are adjoint trees. We shall denote a $\mu$-Bienaym\'e-Galton-Watson tree by $\mathcal T_c$ and an adjoint tree by $\widetilde {\mathcal T}_c$. 

\subsection{Tree-indexed random walk, Green's function, and Branching capacity}\label{sec.treeRW}
We define the random walk $(S_u^x)_{u\in \mathcal T}$ indexed by $\mathcal T$, starting from some $x\in \mathbb Z^d$, as follows. First assign independent and identically distributed random variables
to all the edges of the tree, whose  common law is taken for simplicity to be the uniform distribution on the neighbors of the origin, in other words the jump distribution of the usual simple random walk. It is however likely that most of our results could be proved under much less restrictive hypotheses, in particular being centered, irreducible, and with a finite $d$-th moment as in~\cite{Zhu16} might be sufficient, but we believe that taking simple random walk jumps will ease the reading of some arguments. Then let $S_\varnothing^x =x$ and for each vertex $u$ different from the root define $S_u^x$ as the sum of the random variables assigned to the edges on the unique geodesic going from $u$ to the root. When $x$ is the origin, we shall sometimes drop it from the notation. We denote by 
$$\mathcal T^x = \{S_u^x:u\in \mathcal T\}, \qquad \mathcal T^x_-= \{S_u^x : u \in \mathcal T_-\}, \qquad \mathcal T_+^x=\{S_u^x : u\in \mathcal T_+\},$$
respectively the range of the walk indexed by $\mathcal T$, starting from $x$, and the restriction of its range to the vertices on $\mathcal T_-$ and $\mathcal T_+$.

We define similarly the random walk indexed by a critical tree $\mathcal T_c$ (also called branching random walk), or an adjoint tree $\widetilde{\mathcal T}_c$, and denote their ranges respectively by $\mathcal T^x_c$ and $\widetilde{\mathcal T}_c^x$, when it starts from $x$.

We then let for $z\in \mathbb Z^d$, 
$$G(z) = \mathbb E\Big[ \sum_{u\in\mathcal T_-} \mathbf 1\{S_u = z\}\Big],$$
where we recall $(S_u)_{u\in \mathcal T}$ is the simple random walk indexed by $\mathcal T$ starting from the origin. We also recall that as $\|z\|\to \infty$, one has 
\begin{equation}\label{asympG}
G(z)\sim a_d \cdot \|z\|^{4-d},
\end{equation}
for some constant $a_d>0$, see e.g.~\cite{ASS23}, where $f\sim h$ means that the ratio of the two functions $f$ and $h$ converges to one. Letting now $g$ denote the Green's function of a usual simple random walk $(S_n)_{n\ge 0}$, i.e. $g(z) = \mathbb E[\sum_{n\ge 0} \mathbf 1\{S_n=z\}]$, we recall that one has $g(z)\sim a'_d \cdot \|z\|^{2-d}$, for some other constant $a'_d>0$, see~\cite{LL}. Moreover, it readily follows from the definition of $\mathcal T_-$ that 
\begin{equation}\label{g*g}
G(z) \sim \frac{\sigma^2}{2}\cdot g*g(z),  
\end{equation}
where $g*g(z) = \sum_{y\in \mathbb Z^d}g(z-y)g(y)$.

Now given a finite set $K\subset \mathbb Z^d$, and $x\in K$, we let 
$$e_K(x) = \mathbb P(\mathcal T^x_- \cap K = \emptyset),$$
and define $\textrm{BCap}(K) = \sum_{x\in K} e_K(x)$.

\subsection{Branching interlacements point process}\label{sec.BI}
We define here Branching interlacements using the representation as a Poisson point process on the space of doubly infinite trajectories  modulo time-shift. Let 
$$W = \{w : \mathbb Z \to \mathbb Z^d : \|w(n)\|\to +\infty, \textrm{ as }n\to \pm \infty\}. $$  
This space is naturally endowed with a sigma-algebra $\mathcal W$, generated by the projections $\pi_n:W\to \mathbb Z$, defined by $\pi_n(w)= w(n)$, for $n\in \mathbb Z$. 
We define the shift operators $\theta_k$ on $W$, $k\in \mathbb Z$,  by 
$$\theta_k(w)(n) = w(n+k), \quad \textrm{for all } n\in  \mathbb Z.$$ 
We next define the equivalence relation $\sim$ on $W$ by 
$$w\sim w' \quad \Longleftrightarrow \quad \exists k\in \mathbb Z : w' = \theta_k(w), $$
and let $W^* = W/_\sim$, the quotient space, as well as $\pi_*: W \to W^*$, the canonical projection. Then $W^*$ is naturally endowed with the pushforward sigma-algebra $\mathcal W^*$, induced by this projection. 
For $K\subset \mathbb Z^d$, nonempty and finite, we define $W_K$ the subset of $W$, defined by 
$$W_K = \{w\in W  : \pi_n(w) \in K, \textrm{ for some }n\in \mathbb Z\},$$
the space of trajectories intersecting $K$.
We also define the first entrance time of $w\in W_K$, by 
$$H_K(w) = \inf\{n\in \mathbb Z : w(n) \in K\}. $$ 
Let $W_K^0 = \{w\in W_K : H_K(w) = 0\}$, and $\pi_K: W_K \to W_K^0$, such that $\pi_K(w)$ is the unique element of $W_K^0$ in the same equivalent class as $w$. 
We can now define a family of finite measures $Q_K$, for $K\subset \mathbb Z^d$, finite and nonempty, which are supported on $W_K^0$, and such that for any $A\in \mathcal W$,
\begin{equation}\label{def.QK}
Q_K( A)  
 = \sum_{x\in K} \mathbb P\Big((S_{\mathcal T}^x(n))_{n\in \mathbb Z}  \in A,\,  \mathcal T^x_-\cap K=\emptyset\Big),
\end{equation} 
where we denote by $S^x_{\mathcal T}(n)$ the value of the walk indexed by $\mathcal T$, starting from $x$, at the vertex with label $n$. 
It has been observed in~\cite{Zhu18}, that these measures are consistent in the following sense (for the sake of completeness we reproduce the short proof of the next result below).  
\begin{prop}\label{prop.QK}
For any $K\subset K'\subset \mathbb Z^d$,  finite and nonempty, and any $A\in \mathcal W^*$, 
$$Q_K(\pi_*^{-1}(A) \cap W_K) = Q_{K'}(\pi_*^{-1}(A) \cap W_{K'}). $$ 
\end{prop}
\begin{proof}
Let $A\in \mathcal W^*$, and $B= \pi_K(\pi_*^{-1}(A))$. Since $Q_K$
is supported on $W_K^0$, one has $Q_K(\pi_*^{-1}(A) \cap W_K) = Q_K(B)$, and similarly $Q_{K'}(\pi_*^{-1}(A) \cap W_{K'}) = Q_{K'}(\pi_{K'}(B))$. Hence, it amounts to show that 
\begin{equation}\label{QKB}
Q_K(B)= Q_{K'}(\pi_{K'}(B)).
\end{equation}
For $x\in K'$, $y\in K$ and $k\in \mathbb N$, let 
$$B_{x,y,k} = B\cap \{w(0)=x, H_{K'}(w) = -k, w(-k) = y\}.$$
Note that 
$$\pi_{K'}(B_{x,y,k})=\theta_{-k}(B_{x,y,k})  = \theta_{-k}(B) \cap \{w(0)= y, w(k)=x, H_K(w)=k\}.$$ 
Now by~\eqref{def.QK} and translation invariance of the infinite invariant tree, one has 
\begin{align*}
Q_K(B_{x,y,k}) & = \mathbb P\big(\mathcal S^x_{\mathcal T}(-k) = y, S_{\mathcal T}^x(-k')\notin K' \ \forall k'> k, (S_{\mathcal T}^x(n))_{n\in \mathbb Z} \in B\Big)\\
& = \mathbb P\Big(S_{\mathcal T}^y(k) = x, \mathcal T^y_-\cap K'=\emptyset, S_{\mathcal T}^y(k')\notin K \ \forall k'<k, (S_{\mathcal T}^y(n))_{n\in \mathbb Z} \in \theta_{-k}(B)\Big)\\
& = Q_{K'}\big(\pi_{K'}(B_{x,y,k})\big). 
\end{align*}
Summing over all possible $x$, $y$ and $k$, we get~\eqref{QKB}. 
\end{proof}
As a consequence of Proposition~\ref{prop.QK} one can define a sigma-finite measure $\nu$ on $(W^*,\mathcal W^*)$ by setting 
$$\nu(A) = \limsup_{|K|\to \infty}\  Q_K(\pi_*^{-1}(A)\cap W_K), \qquad \textrm{for any }A\in \mathcal W^*. $$
Importantly, letting $W_K^* = \pi_*(W_K)$, one has
$$\nu(A) = Q_K(\pi_*^{-1}(A)), \qquad \textrm{for any }A\subseteq W_K^*.$$ 
Then the \textit{Branching Interlacement point process} is the Poisson point process on $W^*\times [0,\infty)$ with intensity measure $\nu\times \lambda$, where $\lambda$ is the Lebesgue measure. This is a probability measure $\mathbb Q$ on the space  
$$\Omega = \Big\{\omega=\sum_{n\in \mathbb N} \delta_{(w_n,u_n)} : \omega(W_K^*\times [0,u])<\infty, \textrm{ for any finite }K\subset \mathbb Z^d, \textrm{ and any } u\ge 0\Big\},$$
of locally finite point measures on $W^*\times [0,\infty)$. Then for any $u\ge 0$, we can define the random set $\mathcal I^u\subset \mathbb Z^d$, called \textit{Branching interlacements at level $u$}, by 
$$\mathcal I^u = \bigcup_{n\in \mathbb N:u_n\le u} \textrm{Range}(w_n), $$
where $\sum_{n\in \mathbb N} \delta_{(w_n,u_n)}$ has law $\mathbb Q$, and for $w^*\in W^*$, we write  $\textrm{Range}(w^*) = \{w(n): n\in \mathbb Z\}$, with $w$ any element of $W$ satisfying $\pi_*(w)=w^*$. Note that  by definition the total mass of $Q_K$ is equal to $\textrm{BCap}(K)$, recall~\eqref{def.QK}, and thus we recover~\eqref{def.Iu1}.

Given $K\subset \mathbb Z^d$, nonempty and finite, we define the map $s_K:W_K^*\to W_K^0$, by letting $s_K(w^*)$ be the unique element of $\pi_*^{-1}(w^*)$, which is in $W_K^0$, for any $w^*\in W_K^*$. Then for $u\ge 0$, we can define the map $\mu_{K,u}$ on $\Omega$, by  
$$\mu_{K,u}(\omega) = \sum_{n\, :\, w_n\in W_K^*, u_n\le u} \delta_{s_K(w_n)}, \qquad \textrm{for any }\omega=\sum_{n}\delta_{(w_n,u_n)}.$$
Also given a measurable set $A$, and $\mu= \sum_{i\in I} \delta_{w_i}$, with $I$ finite, we write $\mathbf 1\{w\in A\}\cdot \mu = \sum_{i\in I\, :\, w_i\in A} \delta_{w_i}$.

An important representation of the trace of $\mathcal I^u$ on any finite set $K$, which readily follows from its definition and basic properties of Poisson point processes, is the following. Let $N_K^u$ be a Poisson random variable with mean $u\cdot \textrm{BCap}(K)$, and independently of $N_K^u$, let us give for each $i\ge 1$,  independent random walks indexed by $\mathcal T$, with starting points in $K$ chosen independently of each other according to the probability measure $e_K(\cdot)/\textrm{BCap}(K)$, and conditioned on not hitting $K$ in the past. Denote by $\mathcal T^{(i)}$, $i\ge 1$, their respective ranges. Then 
\begin{equation}\label{rep.IuK}
\mathcal I^u \cap K\  \stackrel{\textrm{(law)}}{=} \ \bigcup_{i=1}^{N_K^u} \mathcal T^{(i)}\cap K.
\end{equation}

\subsection{Hitting probabilities for branching random walks} 
We state here some important results proved in~\cite{Zhu16}.
\begin{prop}[\cite{Zhu16}] \label{Zhu.hit}
There exists a constant $C>0$, such that for any finite set $K\subset \mathbb Z^d$ containing the origin, and any $z\in \mathbb Z^d$, with $\|z\|\ge 2 \cdot \rm{diam}(K)$, 
$$\mathbb P(\mathcal T^z \cap K \neq  \emptyset) \le C\cdot \frac{\rm{BCap}(K)}{ \|z\|^{d-4}}, \quad \text{and} \quad  
\mathbb P(\mathcal T^z_c \cap K \neq  \emptyset) \le C\cdot \frac{\rm{BCap}(K)}{ \|z\|^{d-2}}. $$
Moreover, 
$$\lim_{\|x\|\to \infty} \frac{\mathbb P(\mathcal T^x_- \cap K \neq \emptyset)}{G(x)} = {\rm BCap}(K).$$
\end{prop}
We shall need the following improvement of the first part of the previous proposition, in which we get rid of the hypothesis on the distance between the starting point of the tree-indexed random walk and the set to be hit.  
\begin{cor}\label{cor.hit}
There exists a constant $C>0$, such that for any finite and nonempty set $K\subset \mathbb Z^d$, and any $z\in \mathbb Z^d$, 
$$\mathbb P(\mathcal T^z_- \cap K \neq  \emptyset) \le C\cdot \frac{\rm{BCap}(K)}{ {\rm dist}(z,K)^{d-4}}, \quad \text{and} \quad \mathbb P(\mathcal T^z_c \cap K\neq \emptyset) \le C \cdot \frac{\rm{BCap}(K)}{ {\rm dist}(z,K)^{d-2}}. $$ 
\end{cor}
\begin{proof}
Note that one can always assume that the distance from $z$ to $K$ is larger than $2$, as otherwise the result is straightforward. Define now for $i\ge 0$, $K_i = K \cap \{x : 2^i {\rm dist}(z,K)\le \|z-x\| < 2^{i+1}{\rm dist}(z,K)\}$. Each $K_i$ can be subdivided in at most $100^d$ subsets, whose diameters are smaller than half their distance to $z$. Hence, applying Proposition~\ref{Zhu.hit} to each of these subsets, and for each $i$, yields for some positive constants $C$ and $C'$,  
\begin{equation*}
\mathbb P(\mathcal T^z_- \cap K\neq \emptyset)\le \sum_{i\ge 0} \mathbb P(\mathcal T^z_- \cap K_i\neq \emptyset) \le 100^dC \sum_{i\ge 0} \frac{\rm{BCap}(K)}{2^{i(d-4)} {\rm dist}(z,K)^{d-4}}\le C'\cdot \frac{\rm{BCap}(K)}{ {\rm dist}(z,K)^{d-4}}. 
\end{equation*}
The same proof applies for the walk indexed by $\mathcal T_c$. 
\end{proof}
Another consequence is the following strict monotonicity result for the branching capacity. 
\begin{lem}\label{strict.bcap}
For any finite and nonempty $K\subsetneq K'$, such that there exist $z\in K'\setminus K$ and a path starting from $z$ up to infinity that avoids $K'$, one has
$${\rm BCap}(K) < {\rm BCap}(K'). $$
In particular, there exists $\delta>0$, such that  $\inf_{z\neq 0} \rm{BCap}(\{0,z\})\ge (1+\delta)\cdot \rm{BCap}(\{0\})$. 
\end{lem}
\begin{proof}
First, one can always assume that $K'= K\cup\{z\}$. Next, by using shift-invariance of $\mathcal T$, one has 
\begin{align*}
\rm{BCap}(K') &  = e_{K'}(z) + \sum_{x\in K}\big(e_K(x) - \mathbb P(\mathcal T^x_-\cap K = \emptyset, \mathcal T^x_- \cap \{z\} \neq \emptyset)\big)\\
& = \rm{BCap}(K) + e_{K'}(z) - \sum_{x\in K} \mathbb P(\mathcal T^z_- \cap K'= \emptyset, \mathcal T^z_+ \text{ first hits }K\text{ at }x) \\
& = \rm{BCap}(K) + \mathbb P(\mathcal T^z_-\cap K'=\emptyset, \mathcal T^z_+\cap K = \emptyset). 
\end{align*}
Hence, it amounts to see that the second term above is positive. For this note that it suffices to ask the walk indexed by the spine starting from $z$ follows the (deterministic) path avoiding $K'$, say $\gamma$, until it reaches a distance $R$ from $K'$ large enough, and furthermore that all children of the vertices on the spine until this time have no descendent and have their position also on $\gamma$. Note that this happens with positive probability, and once this is achieved, if $R$ is chosen large enough, both the past and future after this may also avoid $K'$ with positive probability, by using the first part of Proposition~\ref{Zhu.hit}. This concludes the first part of the lemma. The second part follows using  Lemma~\ref{lem.Bcap0x} below and a compacity argument. 
\end{proof}

Given $K\subset \mathbb Z^d$, $n\ge 0$, and a path $\gamma:\{0,\dots,n\}\to \mathbb Z^d$, we say that $\gamma$ goes from $x$ to $K$, and write it as $\gamma:x\to K$, if 
$\gamma(0)=x$, $\gamma(i) \notin K$, for all $i<n$ and $\gamma(n) \in K$. Then we define for such path, 
$$b_K(\gamma) = s(\gamma) \prod_{i=0}^{n-1} k_K(\gamma(i)),$$
where $s(\gamma)$ is the probability that a simple random walk follows the path $\gamma$ during its first $|\gamma|$ steps, i.e. $s(\gamma) = (2d)^{-n}$, and $k_K(x)$ is the probability that a walk indexed by an adjoint tree $\widetilde {\mathcal T}_c$ starting from $x$ does not hit $K$. We also denote by $\Gamma$ the path of first visit to the set $K$, which is the path followed by the walk indexed by the geodesic going from the root to the first vertex in the depth-first search order at which the walk indexed by $\mathcal T_c$ hits $K$. A fundamental result of Zhu~\cite{Zhu16}, says the following.
\begin{prop}[\cite{Zhu16}] \label{Zhu.hit2}
For any $K\subset \mathbb Z^d$, any $x\in \mathbb Z^d$, and any path $\gamma:x\to K$, one has 
$$\mathbb P(\mathcal T^x_c \cap K \neq \emptyset, \Gamma = \gamma) = b_K(\gamma). $$ 
In particular $\sum_{\gamma:x\to K} b_K(\gamma) = \mathbb P(\mathcal T_c^x \cap K \neq \emptyset)$. 
\end{prop}
\begin{rem}\emph{
In fact we shall use more precisely that conditionally on $\Gamma$, the trees attached to the left and to the right of $\Gamma$ are distributed as adjoint trees, those on the left being also conditioned on not hitting $K$. }
\end{rem}

\section{Proof of Proposition~\ref{prop.cov}} \label{sec:prop.cov}
The proposition follows from the next lemma. 
\begin{lem}\label{lem.Bcap0x}
There exists a constant $c>0$, such that 
\[  2\rm{BCap}(\{0\}) - \rm{BCap}(\{0,x\}) \sim \frac{c}{\|x\|^{d-4}}. \]
More specifically, one has $c = 2{\rm BCap}(\{0\})^2 \cdot a_d$. 
\end{lem} 
\begin{proof}
By symmetry ${\rm BCap}(\{0,x\}) = 2\,  e_{\{0,x\}}(0)$. Now one has 
$$e_{\{0,x\}}(0) = {\rm BCap}(\{0\}) - \mathbb P(\mathcal T^0_- \cap \{0\} = \emptyset, \mathcal T^0_- \cap \{x\} \neq \emptyset). $$  
Fix $r=\|x\|^{2/3}$. Let $(\Gamma_n)_{n\ge 0}$ denote the simple random walk starting from the origin, indexed by the spine of $\mathcal T_-$, and  
$$\tau_r = \inf \{n: \|\Gamma_n\|\ge r\}. $$  
Given $n\le m$, let $\mathcal F_-^0[n,m]$ be the part of the range of $\mathcal T^0_-$ associated to the vertices on the critical trees in the past attached to the spine between the intrinsic times $n$ and $m$. 
One has 
\begin{align}\label{decomposition1}
\nonumber & |\mathbb P\big(\mathcal T^0_- \cap \{0\} = \emptyset, \mathcal T^0_- \cap \{x\} \neq \emptyset\big) - 
\mathbb P\big(\mathcal F^0_-[0,\tau_r] \cap \{0\} = \emptyset, \mathcal F^0_-[\tau_r,\infty) \cap \{x\} \neq \emptyset\big)| \\
& \le \mathbb P\big(\mathcal F^0_-[\tau_r,\infty) \cap \{0\} \neq  \emptyset, \mathcal F^0_-[\tau_r,\infty) \cap \{x\} \neq \emptyset\big) + 
2\mathbb P\big(\mathcal F^0_-[0,\tau_r] \cap \{x\} \neq \emptyset\big). 
\end{align}
Now by Proposition~\ref{Zhu.hit},  
\begin{align*}
\mathbb P(\mathcal F^0_-[0,\tau_r] \cap \{0\} = \emptyset, \mathcal F^0_-[\tau_r,\infty) \cap \{x\} \neq \emptyset) 
& \sim  \mathbb P(\mathcal F^0_-[0,\tau_r] \cap \{0\} = \emptyset) \cdot {\rm BCap}(\{0\})\cdot G(x) \\   
& \sim {\rm BCap}(\{0\})^2 \cdot G(x). 
\end{align*}
Hence, it just remains to show that the two terms on the right-hand side of~\eqref{decomposition1} are of negligible order. Concerning the second one, one has by Proposition~\ref{Zhu.hit}, 
$$\mathbb P\big(\mathcal F^0_-[0,\tau_r] \cap \{x\} \neq \emptyset\big)\lesssim \mathbb E[\tau_r] \cdot \|x\|^{2-d} 
\lesssim \|x\|^{2-d+4/3}. 
$$
On the other hand, concerning the first term on the right-hand side of~\eqref{decomposition1}, it amounts to show that uniformly over $z\in \partial B(0,r)$, 
\begin{equation}\label{goal}
\mathbb P\big(\mathcal T^z_- \cap \{0\} \neq  \emptyset,\,  \mathcal T^z_- \cap \{x\} \neq \emptyset\big) = o(G(x)). 
\end{equation}
Consider again $(\Gamma_n)_{n\ge 0}$ the random walk indexed by the spine of $\mathcal T_-$. Applying again Proposition~\ref{Zhu.hit} yields
$$\mathbb P\big(\mathcal T^z_- \cap \{0\} \neq  \emptyset,\,  \mathcal T^z_- \cap \{x\} \neq \emptyset\big) 
\lesssim \sum_{n\neq m} \mathbb E_z[g(\Gamma_n)g(\Gamma_m-x)] 
+ \sum_{n\ge 0} \sum_{y\in \mathbb Z^d} \mathbb E_z[g(\Gamma_n-y)] \cdot g(x-y)g(y), 
$$ 
using the many-to-two lemma for the second term. Now elementary computation, using in particular~\eqref{g*g}, give 
\begin{align*}
& \sum_{n\ge 0} \sum_{y\in \mathbb Z^d} \mathbb E_z[g(\Gamma_n-y)] \cdot g(x-y)g(y)
= \sum_{u,y\in \mathbb Z^d} g(z-u)g(u-y)g(x-y)g(y) \\
& \lesssim \sum_{y\in \mathbb Z^d} G(z-y) g(x-y) g(y) \lesssim G(z) \cdot G(x), 
\end{align*}
and 
\begin{align*}
 \sum_{n\neq m} \mathbb E_z[g(\Gamma_n)g(\Gamma_m-x)] 
& \lesssim \sum_{u,v\in \mathbb Z^d} \Big(g(z-u) g(u) g(v-u) g(v-x)
+ g(z-u)g(u-x)g(v-u)g(v) \Big)\\
& \lesssim \sum_{u\in \mathbb Z^d} g(z-u)g(u) G(u-x) + g(z-u)g(u-x)G(u) \\
& \lesssim G(z) \cdot G(x),
\end{align*}
which altogether proves well~\eqref{goal}, and concludes the proof of the lemma. 
 \end{proof}

\begin{proof}[Proof of Proposition~\ref{prop.cov}] 
One has 
\begin{align*}
{\rm Cov} (\mathbf 1_{x\in \mathcal V^u}, \mathbf 1_{y\in \mathcal V^u}) & = \mathbb P(\{x,y\}\subset  \mathcal V^u) - \mathbb P(x\in \mathcal V^u) \cdot \mathbb P(y\in \mathcal V^u)   \\
 & =  \exp\big(-u\textrm{BCap}(\{x,y\})\big)-\exp\big(-2u\textrm{BCap}(\{0\})\big)  \\
  & \sim  \frac{cu}{\|x-y\|^{d-4}}\cdot \exp\big(-2u\textrm{BCap}(\{0\})\big)  ,
\end{align*}
with $c$ as in Lemma~\ref{lem.Bcap0x}. 
\end{proof}

\section{Proof of Theorem~\ref{thm.cov}}
\label{sec:thm.cov}
The proof is based on the following lemma. 
\begin{lem} \label{lem.AB}
There exist positive constants $c$ and $C$, such that for any finite and nonempty $K_1,K_2\subset \mathbb Z^d$, 
satisfying 
\begin{equation}\label{hyp.K1K2}
\rm{dist}(K_1,K_2)\ge C\cdot \rm{diam}(K_1),
\end{equation}
one has 
\begin{align}\label{eq.lem.AB}
 \sum_{x\in K_1} \mathbb P(\mathcal T^x_- \cap K_1 = \emptyset, \mathcal T^x_+ \cap K_2\neq \emptyset)
 \le C   \left\{   \frac{\rm{BCap}(K_1) \cdot \rm{BCap}(K_2)}{\rm{dist}(K_1,K_2)^{d-4}} 
+   \rm{BCap}(K_2)\cdot e^{-c\frac{\rm{dist}(K_1,K_2)}{\rm{diam}(K_1)} } \right\}. 
\end{align}
\end{lem}
\begin{proof}
Assume without loss of generality that $0\in K_1$. Let $R=\textrm{diam}(K_1)$, and let $D\ge 1$ be large enough, so that 
\begin{equation}\label{escape}
\mathbb P(\mathcal T^z_- \cap K_1 \neq \emptyset) \le 1/2, \qquad \text{for any }z\in \partial B(0,DR).
\end{equation}
Note that $D$ can be chosen independently of $K_1$, by Proposition~\ref{Zhu.hit}. 
Now fix $x\in K_1$, and denote by $\Gamma=(\Gamma_n)_{n\ge 0}$ the simple random walk starting from $x$, indexed by the set of vertices on the spine of $\mathcal T_-$.
Define 
$$\tau= \inf \{n \ge 0 : \Gamma_n \in \partial B(0,2DR)\},$$
and let $\tau'$ be the last time $\Gamma$ is in the ball $B(0,DR)$ before time $\tau$, that is formally
$$\tau' = \sup\{ n \le \tau : \Gamma_n \in B(0,DR)\}.$$
For $t>0$, let $\mathcal F_+^x[0,t]$ be the forest of critical trees in the future of $\mathcal T^x$, which are rooted at the vertices on the spine before time $t$, and similarly for $\mathcal F_-^x[0,t]$. Assume that the constant $C$ in~\eqref{hyp.K1K2} is large enough, so that $\rm{dist}(K_1,K_2) \ge 3DR$. Then by Corollary~\ref{cor.hit}, and using also~\eqref{escape} at the third line, one has   
\begin{align*}
&  \mathbb P(\mathcal T^x_- \cap K_1 = \emptyset, \mathcal T^x_+ \cap K_2\neq \emptyset)  \le  \mathbb P(\mathcal F^x_-[0,\tau] \cap K_1 = \emptyset, \mathcal T^x_+ \cap K_2\neq \emptyset)\\
& \le  \mathbb P(\mathcal F^x_-[0,\tau] \cap K_1 = \emptyset, \mathcal F^x_+[0,\tau] \cap K_2\neq \emptyset) + C \, \mathbb P(\mathcal F^x_-[0,\tau] \cap K_1 = \emptyset)\cdot \frac{\rm{BCap}(K_2)}{\rm{dist}(K_1,K_2)^{d-4}}\\
& \lesssim  \mathbb E\Big[\tau \cdot \mathbf 1\{\mathcal F^x_-[0,\tau'] \cap K_1 = \emptyset\}\Big] \cdot  \frac{\rm{BCap}(K_2)}{\rm{dist}(K_1,K_2)^{d-2}} +  e_{K_1}(x) \cdot \frac{\rm{BCap}(K_2)}{\rm{dist}(K_1,K_2)^{d-4}}. 
\end{align*} 
Note that after summing over $x\in K_1$, the second term on the right hand side will be bounded by the first term in the right-hand side of~\eqref{eq.lem.AB}. Thus it only remains to consider the first term in the right-hand side above. 
By Lemma 3.1 in~\cite{ASS23}, there exists a constant $c>0$, such that for any $w\in \partial B(0,DR)$,  
$$ \mathbb E_w\left[ H_{\partial B(0,2DR)} \cdot \mathbf 1\{H_{\partial B(0,2DR)} < H_{B(0,DR)}^+\}\right]  \le 
c\, R^2\cdot  \mathbb P_w(H_{\partial B(0,2DR)} < H_{B(0,DR)}^+).$$
Together with~\eqref{escape} and the Markov property this yields for any $x\in K_1$, 
$$
 \mathbb E\Big[\tau \cdot \mathbf 1\{\mathcal F^x_-[0,\tau'] \cap K_1 = \emptyset\}\Big] \le \mathbb E\Big[\tau' \cdot \mathbf 1\{\mathcal F^x_-[0,\tau'] \cap K_1 = \emptyset\}\Big] + 2 c\, e_{K_1}(x) R^2.
$$
Since $R^2 \lesssim \textrm{dist}(K_1,K_2)^2$, it again only remains to consider the first term in the right-hand side above. 
Relying once more on~\eqref{escape}, yields  
$$\mathbb E\Big[\tau' \cdot \mathbf 1\{\mathcal F^x_-[0,\tau'] \cap K_1 = \emptyset\}\Big]  \le 2 \textrm{dist}(K_1,K_2)^2\cdot e_{K_1}(x)  + 
\mathbb E\Big[\tau' \cdot \mathbf 1\{\tau'\ge \textrm{dist}(K_1,K_2)^2, \mathcal F^x_-[0,\tau'] \cap K_1 = \emptyset\}\Big],$$
and one is left to bound the second term on the right-hand side above. 
Observe first that by choosing $D$ large enough, one can always ensure that for any $w \in \partial B(0,DR)$, 
\begin{equation}\label{tildee}
\mathbb P_w(H_{\partial B(0,2DR)} < H_{B(0,DR)}^+)\le 2 \, \widetilde e_{B(0,DR)}(w). 
\end{equation}
Now, letting $k_0 = \lfloor \tfrac{\textrm{dist}(K_1,K_2)^2}{R^2}\rfloor$, one has 
$$\mathbb E\Big[\tau' \cdot \mathbf 1\{\tau'\ge \textrm{dist}(K_1,K_2)^2, \mathcal F^x_-[0,\tau'] \cap K_1 = \emptyset\}\Big]
\lesssim R^2\cdot \sum_{k\ge k_0} \mathbb P\Big(\tau'\ge k R^2, \mathcal F^x_-[0,\tau'] \cap K_1 = \emptyset\Big). $$ 
Then for any $k\ge k_0$, by reversing time on the spine, and using Proposition~\ref{Zhu.hit2}, one can write using also~\eqref{tildee} 
for the last inequality, 
\begin{align*}
&  \sum_{x\in K_1} 
\mathbb P\Big(\tau'\ge k R^2, \mathcal F^x_-[0,\tau'] \cap K_1 = \emptyset\Big)  = \sum_{w\in \partial B(0,DR)} \sum_{x\in K_1}  \mathbb P\Big(\tau'\ge k R^2, \mathcal F^x_-[0,\tau'] \cap K_1 = \emptyset, \Gamma(\tau')= w\Big) \\
 &\le 2 \sum_{w\in \partial B(0,DR)} \widetilde e_{B(0,DR)}(w) \cdot \mathbb P\big(\exists u \in \mathcal T_c : |u| \ge kR^2, \, S^w_u\in  K_1, \, S^w_v\in B(0,2DR) \textrm{ for all }v\le u\big), 
\end{align*}
where in the last probability $(S^w_u)_{u\in \mathcal T_c}$ denotes a walk indexed by a critical tree $\mathcal T_c$, starting from $w$, and $v\le u$ means that $v$ is a vertex on the geodesic between the root of $\mathcal T_c$ and $u$. Conditioning next on the tree up to generation $(k-1)R^2$, and using~\eqref{Kolmogorov}, we get that for any $w\in \partial B(0,DR)$, 
\begin{align*}
& \mathbb P\big(\exists u \in \mathcal T_c : |u| \ge kR^2, \, S^w_u\in  K_1, \, S^w_v\in B(0,2DR) \textrm{ for all }v\le u\big) \\
& \lesssim \frac {\mathbb P_w(H_{\partial B(0,2DR)} \ge (k-1)R^2)}{R^2} \cdot \mathbb E[Z_{(k-1)R^2}]  
\lesssim  \frac 1{R^2} \cdot \exp(- c k),
\end{align*}
for some constant $c>0$. Altogether, and using that the Newtonian capacity of $B(0,DR)$ is of order $R^{d-2}$, we obtain~\eqref{eq.lem.AB}, as wanted. 
\end{proof}

\begin{proof}[Proof of Theorem~\ref{thm.cov}] 
The firt part of the argument is the same as in~\cite{Sz10}, so let us briefly recall the main lines. Let $K = K_1 \cup K_2$. First observe that 
$$\mu_{K,u} = \mu_{1,1} + \mu_{1,2} + \mu_{2,1} + \mu_{2,2}, $$ 
where 
$$\mu_{1,1} = \mathbf 1\{w(0)\in K_1, H_{K_2}(w) =\infty\} \cdot \mu_{K,u}, \quad \mu_{1,2} = \mathbf 1\{w(0)\in K_1, H_{K_2}(w) <\infty\} \cdot \mu_{K,u}, $$
with similar formula for $\mu_{2,2}$ and $\mu_{2,1}$. In particular $\mu_{1,1}$, $\mu_{1,2}$, $\mu_{2,2}$ and $\mu_{2,1}$ are independent Poisson point processes, and by definition the events $E$ and $F$ are measurable  functions of these point processes. More precisely, one can find measurable functions $f$ and $g$, such that 
$$\mathbf 1_E= f(\mu_{1,1} + \mu_{1,2} + \mu_{2,1}), \quad \textrm{and}\quad  
\mathbf 1_F = g(\mu_{2,2} + \mu_{1,2} + \mu_{2,1}). $$  
Consequently, one has 
$$|\textrm{Cov}(E,F)|\le 2 (\mathbb Q( \mu_{1,2} \neq 0) + \mathbb Q(\mu_{2,1}\neq 0)),$$
where we recall that $\mathbb Q$ denotes the law of the Branching Interlacements point process. Moreover, by definition one has 
\begin{align*}
\mathbb Q(\mu_{1,2}\neq 0) & = 1 - e^{-u\cdot  Q_{K_1}(0\le H_{K_2}(w)<\infty)} \le u \cdot Q_{K_1}(0\le H_{K_2}(w)<\infty) \\
&  \le u \sum_{x\in K_1} \mathbb P(\mathcal T_-^x\cap K_1 = \emptyset, \mathcal T^x_+ \cap K_2 <\infty). 
\end{align*}
Note now that one can always assume that~\eqref{hyp.K1K2} is satisfied, as otherwise the result is straightforward. 
Hence we can apply Lemma~\ref{lem.AB}, which provides the desired bound for the sum above. Similarly, one has 
$$\mathbb Q(\mu_{2,1}\neq 0) \le u \sum_{x\in K_2} \mathbb P(\mathcal T_-^x\cap K_2 = \emptyset, \mathcal T^x_+ \cap K_1 <\infty). $$ 
Now assume without loss of generality that $0\in K_1$, and for $i\ge 0$, let 
$$K_{2,i}= \{z\in K_2: 2^i \rm{dist}(K_1,K_2) \le \|z\| \le 2^{i+1} \rm{dist}(K_1,K_2)\}.$$ 
For each $i$, $K_{2,i}$ can be further subdivided in a finite number of subsets, say $K_{2,i,j}$, $j\in J$, with $|J|\le C_0^d$, for some constant $C_0>0$, such that $\rm{dist}(K_1,K_{2,i,j}) \ge C \cdot \rm{diam}(K_{2,i,j})$, for any $i,j$. Then we can write, using Lemma~\ref{lem.AB}, 
\begin{align*}
& \sum_{x\in K_2} \mathbb P(\mathcal T_-^x\cap K_2 = \emptyset, \mathcal T^x_+ \cap K_1 <\infty) \le \sum_{i\ge 0}\sum_{j\in J} \sum_{x\in K_{2,i,j}} \mathbb P(\mathcal T_-^x\cap K_{2,i,j} = \emptyset, \mathcal T^x_+ \cap K_1 <\infty)\\
 & \le C C_0^d \sum_{i\ge 0} \left\{
  \frac{\rm{BCap}(K_1) \cdot \rm{BCap}(K_2)}{2^{i(d-4)}\rm{dist}(K_1,K_2)^{d-4}} 
+   \rm{BCap}(K_2)\cdot e^{-c2^i\frac{\rm{dist}(K_1,K_2)}{\rm{diam}(K_1)} } \right\} \\ 
& \lesssim \frac{\rm{BCap}(K_1) \cdot \rm{BCap}(K_2)}{\rm{dist}(K_1,K_2)^{d-4}} 
+   \rm{BCap}(K_2)\cdot e^{-c\frac{\rm{dist}(K_1,K_2)}{\rm{diam}(K_1)} }, 
\end{align*}
concluding the proof of the theorem. 
\end{proof}

\section{Proof of Theorem~\ref{thm:pt}} 
\label{sec:thm:pt}
We prove here Theorem~\ref{thm:pt}. For the proof of the upper bound $u_*<\infty$, one can simply adapt the arguments of~\cite{Sz10}, as we explain in Section~\ref{subsec:u<infty} below. The main difficulty arises in the proof of the lower bound $u_*>0$. Actually, in dimension $d\ge 9$, one could also rely on the relatively soft arguments of~\cite{Sz10}, but the lower dimensional cases are substantially more difficult. More specifically, we use the same scheme of proof as in~\cite{R}, which is partly inspired by~\cite{Sz12}, but as mentioned in the introduction the whole matter in our setting is to prove that some sufficiently high exponential moment of the number of boxes visited by a tree-indexed random walk is finite. In the case of a simple random walk, this followed from the fact that this random variable could be stochastically dominated by a geometric random variable using the Markov property of random walks. Here we have to use instead some intricate inductive analysis, which is conducted in Section~\ref{sec:u>0}. First, in Section~\ref{sec:prel} we gather preparatory results, and recall the main construction from~\cite{R,Sz12}, which consists in a set of boxes organized in a hierarchical manner, that the random set under consideration (either the vacant set in the proof of the upper bound, or the interlacement set for the lower bound)  has to all intersect.

\subsection{Preliminaries and notation}
\label{sec:prel}
First, we recall a result proved in~\cite{ASS23}: there exist positive constants $c$ and $C$ (possibly depending on the dimension), such that for any finite nonempty $K\subset  \mathbb Z^d$, and any $x\in \mathbb Z^d$, 
$$c \le \sum_{y\in K} G(y-x) e_K(y)\le C.$$
By summing these inequalities over $x\in K$, we deduce that for any finite $K$,
\begin{equation}\label{lowerBCap}
\frac{c|K|}{\max_{x\in K} \sum_{y\in K} G(y-x)}\le \textrm{BCap}(K) \le \frac{C|K|}{\min_{x\in K} \sum_{y\in K} G(y-x)}. 
\end{equation}
Now, similarly as in~\cite{R,Sz12} we define  
$T_{(n)}=\{1,2\}^n$ (in particular $T_{(0)} = \emptyset$), and $T_n= \cup_{k=0}^n T_k$,  the binary tree of depth $n$. For $m=(\xi_1,\dots,\xi_k)\in T_{(k)}$, we let 
$m_1 = (\xi_1,\dots,\xi_k,1)$ and $m_2 =(\xi_1,\dots,\xi_k,2)$ be the two children of $m$ in $T_{(k+1)}$. Then let $L_0\ge 1$ be given, and define   
\begin{equation}\label{Ln}
L_n = L_0\cdot 6^n, \quad n\ge 0,
\end{equation}
as well as $\mathcal L_n = L_n\mathbb Z^d$. 
We say that a map 
$\mathcal M : T_n \to \mathbb Z^d$, is a \textbf{proper embedding} of $T_n$ with root at $x\in \mathbb Z^d$, if 
\begin{enumerate}
    \item $\mathcal M(\emptyset)=x$; 
    \item for all $0\le k\le n$ and $m\in T_{(k)}$, we have $\mathcal M(m) \in \mathcal L_{n-k}$; 
    \item for all $0\le k\le n$ and $m\in T_{(k)}$, we have 
    $$|\mathcal M(m_1) - \mathcal M(m)| = L_{n-k}, \quad |\mathcal M(m_2) - \mathcal M(m)| = 2L_{n-k}. $$
\end{enumerate}
We let $\Lambda_{n,x}$ be the set of proper embeddings of $T_n$ into $\mathbb Z^d$ with root at $x$. The three following lemmas are proved in~\cite{R}. 
\begin{lem} \label{lem:count}
Let $d\ge 1$. There exists a constant $\mathcal C_d>0$, such that for any $n\ge 1$ and $x\in \mathbb Z^d$, the number of proper embeddings of $T_n$ into $\mathbb Z^d$ with root at $x$ is equal to $|\Lambda_{n,x}|= \mathcal C_d^{2^n-1}$. 
\end{lem} 
Let $S(x,R)= \{y\in \mathbb Z^d : |y-x| = R\}$, with $|\cdot|$ denoting the $\ell_\infty$-norm.
A path $\gamma : \{0,\dots,\ell\}\to \mathbb Z^d$, is said to be $*$-connected, if $|\gamma(i)- \gamma(i-1)|=1$, for all $i\ge 1$. We let $\mathcal R(\gamma) = \{\gamma(0),\dots,\gamma(\ell)\}$ be its range. 
\begin{lem}\label{lem:*path}
If $\gamma$ is a $*$-connected path in $\mathbb Z^d$, $d\ge 2$, and $x\in \mathcal L_n$, is such that 
$$\mathcal R(\gamma) \cap S(x,L_n-1) \neq \emptyset \quad \textrm{and}\quad \mathcal R(\gamma)\cap S(x,2L_n)\neq \emptyset,$$
then there exists $\mathcal M\in \Lambda_{n,x}$, such that 
$$\mathcal R(\gamma) \cap S(\mathcal M(m),L_0-1)\neq \emptyset  \quad \textrm{for all }m\in T_{(n)}. $$ 
\end{lem}
For $m \in T_{(n)}$, and $1\le k\le n$, let $T_{(n)}^{m,k}$ be the set of elements of $T_{(n)}$ whose least common ancestor with $m$ lies in $T_{(k)}$. 
\begin{lem}\label{lem:Lambdadist}
For any $n\ge 1$, $x\in \mathcal L_n$, $\mathcal M\in \Lambda_{n,x}$, $m\in T_{(n)}$, $k\ge 1$, $m'\in T_{(n)}^{m,k}$, any $y\in S(\mathcal M(m),L_0-1)$ and any $z\in S(\mathcal M(m'),L_0-1)$, one has $|y-z|\ge L_{k-1}$. 
\end{lem}

\subsection{Proof of $u_*<\infty$}\label{subsec:u<infty}
Here we can take $L_0= 1$ in the above construction. Then for $n\ge 1$, define 
$$A_n^u = \left\{
\begin{array}{c}
\textrm{there exists a nearest neighbor path in }\mathcal V^u \\
\textrm{that connects }S(0,L_n-1) \textrm{ to  }S(0,2L_n)
\end{array}\right\}. $$ 
A standard and direct argument of monotone convergence shows that to prove that  $u_*$ is finite, it suffices to prove the next proposition (see e.g.~\cite{R} for details).
\begin{prop}\label{prop.Anu}
Let $d\ge 5$. There exists $u_1<\infty$, such that for any $u>u_1$, there exists $q=q(d,u)\in (0,1)$, such that for any $n\ge 1$, 
$$\mathbb P(A_n^u) \le q^{2^n}. $$ 
\end{prop}
\begin{proof}
For $n\ge 1$, and $\mathcal M\in \Lambda_{n,0}$, 
set $\mathcal X_{\mathcal M} = \cup_{m\in T_{(n)}}\mathcal M(m)$. Thanks to Lemma~\ref{lem:*path}, one has
\begin{align*}
\mathbb P(A_n^u) & \le \mathbb P\big(\bigcup_{\mathcal M\in \Lambda_{n,0}} \{\mathcal X_{\mathcal M} \subset \mathcal V^u\}\big)\stackrel{\eqref{def.Iu1}}{\le} \sum_{\mathcal M\in \Lambda_{n,0}}
\exp\big(-u\cdot \textrm{BCap}(\mathcal X_{\mathcal M})\big)\\
& \le {\mathcal C}_d^{2^n} \cdot \max_{\mathcal M\in \Lambda_{n,0}}
\exp\big(-u\cdot \textrm{BCap}(\mathcal X_{\mathcal M})\big), 
\end{align*}
using also Lemma~\ref{lem:count} for the last inequality. Now~\eqref{asympG} and Lemma~\ref{lem:Lambdadist} yield the existence of a constant $C>0$, such that for any $\mathcal M\in \Lambda_{n,0}$,  
$$\max_{x\in \mathcal X_{\mathcal M}} \sum_{y\in \mathcal X_{\mathcal M} } G(y-x) \le C \sum_{k=1}^n \frac{2^k}{L_{k-1}^{d-4}} \le 2C\sum_{k=1}^n \frac 1{3^{k-1}} \le 3C,$$ 
and hence by~\eqref{lowerBCap}, for some constant $c$ (independent of $\mathcal M$), 
$$\textrm{BCap}(\mathcal X_{\mathcal M}) \ge c\cdot 2^n. $$ 
Taking now $u$ large enough yields the desired estimate and concludes the proof. 
\end{proof}

\subsection{Proof of $u_*>0$}\label{sec:u>0}

Here we consider $F = \mathbb Z^2 \times \{0\}^{d-2}$ a two-dimensional subspace, and given $x\in \mathbb Z^d$, and $r\ge 0$, we let $S_F(x,r) = S(x,r)\cap F$. One reason for considering spheres in $F$ rather than on the whole space is that this way we get smaller sets, which are more difficult to hit by  branching random walks. More specifically,~\eqref{asympG} and~\eqref{lowerBCap} yield the following bounds on their branching capacity: for any $x\in \mathbb Z^d$, and $L_0\ge 1$, 
\begin{equation}\label{BCapSL0}
    \textrm{BCap}(S_F(x,L_0)) \asymp \left\{
    \begin{array}{ll}
    L_0 & \textrm{if }d\ge 6\\
    \frac{L_0}{1+\log(L_0)} & \textrm{if }d=5.
\end{array}
\right. 
\end{equation}
Another reason to work on a two-dimensional space, is that one can use duality. More precisely, by duality the event that there is no  path in the vacant set inside $F$ from say the boundary of $S_F(0,L_n)$ to infinity is the same as the event that $S_F(0,L_n)$ is surrounded by a $*$-path in $\mathcal I^u\cap F$. 
Consequently, it can be seen (see~\cite{R} for details) that to prove that $u_*$ is positive, it is sufficient to prove the next proposition. 
\begin{prop}\label{prop.Bnu}
Define for $x\in \mathcal L_n\cap F$, $n\ge 1$ and $u>0$, \begin{equation*}
B_{n,x}^u =\left\{ \begin{array}{cc}
\textrm{there exists a }*\textrm{-connected path in }\mathcal I^u\cap F\\
\textrm{that connects } S_F(x,L_n-1) \textrm{ to } S_F(x,2L_n) 
\end{array}\right\}. 
\end{equation*}
There exists $u_0>0$ and $\rho \in (0,1)$, such that for any $u<u_0$, any $n\ge 1$ and any $x\in \mathcal L_n \cap F$, 
$$\mathbb P(B_{n,x}^u) \le \rho^{2^n}. $$ 
\end{prop}
The proof of this result relies on Proposition~\ref{prop.main} below, which is the main original contribution of this paper. In order to state it, one needs some additional notation. 
Given $n\ge 1$ and $x\in F$, we denote by
$\Lambda_{n,x}^F$ the set of proper embedding with root at $x$ which take values in $F$, i.e. the set of $\mathcal M\in \Lambda_{n,x}$ such that $\mathcal M(m) \in F$, for all $m\in T_n$. 
Given $\mathcal M\in \Lambda_{n,x}^F$, we further define for any $m\in T_{(n)}$, the frame: 
$$\square_m = S_F(\mathcal M(m),L_0-1),$$
and set 
$$\mathcal X_{\mathcal M}^F = \bigcup_{m\in T_{(n)}} \square_m.$$
Similarly as for~\eqref{BCapSL0}, one can see using~\eqref{asympG} and~\eqref{lowerBCap} that
\begin{equation}\label{BCapSLF}
    \textrm{BCap}(\mathcal X_{\mathcal M}^F) \asymp \left\{
    \begin{array}{ll}
   2^n  L_0  & \textrm{if }d\ge 6\\
    \frac{2^nL_0}{1+\log(L_0)} & \textrm{if }d=5,
\end{array}
\right. 
\end{equation}
with the implicit constants independent of $n\ge 1$, $x\in \mathcal L_n$ and $\mathcal M\in \Lambda_{n,x}^F$. Moreover, for any $A\subseteq \mathbb Z^d$, we let 
$$\mathcal N_{\mathcal M}(A) =\sum_{m\in T_{(n)}} \mathbf 1\{A \cap \square_m\neq \emptyset\}. $$ 
Recall also that given $z\in \mathbb Z^d$, we denote by $\mathcal T^z$ the range of a random walk indexed by $\mathcal T$, starting from $z$, and by $\mathcal T^z_-$ its restriction to the set of vertices in the past. 
\begin{prop}\label{prop.main}
Let $\lambda>0$. There exist  $C>0$, and $L_0\ge 1$, such that for any $n\ge 1$, any $x\in \mathcal L_n\cap F$, and any $\mathcal M\in \Lambda_{n,x}^F$, 
$$\frac 1{\textrm{BCap}(\mathcal X_{\mathcal M}^F)} \sum_{z\in \mathcal X_{\mathcal M}^F} \mathbb E\Big[e^{\lambda \cdot \mathcal N_{\mathcal M}(\mathcal T^z)}\cdot \mathbf 1\{ \mathcal T^z_- \cap \mathcal X_{\mathcal M}^F = \emptyset\}\Big]\le C. $$ 
\end{prop}
Assuming this proposition, one can now give the proof of Proposition~\ref{prop.Bnu}. 
\begin{proof}[Proof of Proposition~\ref{prop.Bnu}]
Using Lemma~\ref{lem:count} and~\ref{lem:*path}, we get that for some constant $C>0$, 
\begin{align}\label{NMIu1}
\nonumber \mathbb P(B_{n,x}^u)&  \le \mathcal C_d^{2^n} \cdot \max_{\mathcal M\in \Lambda_{n,x}^F}\mathbb P\big(\mathcal I^u \cap S_F(\mathcal M(m),L_0-1)\neq \emptyset, \textrm{ for all }m\in T_{(n)}\big) \\  
&  = \mathcal C_d^{2^n} \cdot \max_{\mathcal M\in \Lambda_{n,x}^F}\mathbb P\big(\mathcal N_{\mathcal M}(\mathcal I^u) \ge 2^n\big). 
\end{align}
We use next the representation of $\mathcal I^u$ given by~\eqref{rep.IuK}. Fix any $\mathcal M\in \Lambda_{n,x}^F$.  
Let $N^u$ be a Poisson random variable with mean $u\cdot \textrm{BCap}(\mathcal X_{\mathcal M}^F)$, and consider a sequence of independent random walks $(S^{(i)})_{i\ge 1}$ indexed by $\mathcal T$ starting from randomly and independently chosen points of $\mathcal X_{\mathcal M}^F$ distributed according to $e_{\mathcal X_{\mathcal M}^F}(\cdot)/\textrm{BCap}(\mathcal X_{\mathcal M}^F)$, and conditioned on not hitting $\mathcal X_{\mathcal M}^F$ in the past. Denote by $(\mathcal T^{(i)})_{i\ge 1}$, their respective ranges. Then~\eqref{rep.IuK} yields that for any $\lambda>0$, using also Campbell's formula for the last line, 
\begin{align}\label{NMIu2}
\nonumber    \mathbb P\big(\mathcal N_{\mathcal M}(\mathcal I^u) \ge 2^n\big) & \le  \mathbb P\Big(\sum_{i=1}^{N^u} \mathcal N_{\mathcal M}(\mathcal T^{(i)}) \ge 2^n\Big)
     \le \exp(-\lambda \cdot 2^n) \cdot \mathbb E\Big[\exp\Big(\lambda \sum_{i=1}^{N^u} \mathcal N_{\mathcal M}(\mathcal T^{(i)})\Big)\Big]\\
    & = \exp(-\lambda \cdot 2^n)\cdot \exp\Big(u \cdot \textrm{BCap}(\mathcal X_{\mathcal M}^F) \cdot (\mathbb E[e^{\lambda \cdot \mathcal N_{\mathcal M}(\mathcal T^{(1)})}] - 1)\Big).
\end{align}
Now we use that by~\eqref{BCapSL0} and subadditivity of the branching capacity one has  $\textrm{BCap}(\mathcal X_{\mathcal M}^F) \le  C\cdot 2^n L_0$, for some constant $C>0$ (independent of $\mathcal M$). It then suffices to take $\lambda = 2 \log \mathcal C_d$, and choose $L_0$ large enough so that by Proposition~\ref{prop.main}, the exponential moment of $\mathcal N_{\mathcal M}(\mathcal T^{(1)})$ appearing above is finite. Finally, combining~\eqref{NMIu1} and~\eqref{NMIu2} shows that for $u$ small enough the probability of $B_{n,x}^u$ is exponentially small, concluding the proof of the proposition. 
\end{proof}
It amounts to prove Proposition~\ref{prop.main} now. For this, we need to introduce some more notation. Let $n\ge 1$, $x\in \mathcal L_n\cap F$ and $\mathcal M\in \Lambda_{n,x}^F$ be given. For $z\in \mathbb Z^d$ and $r>0$, we let $Q(z,r) = \{y \in \mathbb Z^d : |y-z|\le r\}$, where we recall that $|\cdot|$ denotes the $\ell_\infty$-norm. We then define inductively a sequence $(\mathcal S_i)_{i\ge 0}$ of subsets of $\mathbb Z^d$ as follows. First 
$$\mathcal S_0 = \bigcup_{m\in T_{(n)}} Q(\mathcal M(m),L_1).$$
Then for $i\in \{1,\dots,n\}$, we let 
$$\mathcal S_i = \Big(\bigcup_{m \in T_{(n-i)}} Q(\mathcal M(m),L_{i+1})\Big)\setminus \mathcal S_{i-1},$$
and for $i>n$, we just consider 
$$\mathcal S_i = Q(x,L_{i+1}) \setminus Q(x,L_i).$$
Note in particular that the sequence $(\mathcal S_i)_{i\ge 0}$ forms a partition of $\mathbb Z^d$. Finally, we consider the random variable defined for $\delta\in (0,1)$ and $z\in \mathbb Z^d$, by  
$$\mathcal N_{\mathcal M}^\delta(z) = \sum_{\substack{m\in T_{(n)} : \\ d(z,\square_m)\ge \delta L_0}} \mathbf 1\{\widetilde{\mathcal T}^z_c \cap \square_m\neq \emptyset\},$$
where we recall that $\widetilde{\mathcal T}^z_c$ denotes the range of a random walk indexed by an adjoint tree starting from $z$, and for $A\subset \mathbb Z^d$, $d(z,A)=\inf\{|z-a| : a\in A\}$ denotes the $\ell_\infty$-distance from $z$ to $A$. In particular if the distance from $z$ to $\mathcal X_{\mathcal M}^F$ is larger than $\delta L_0$, then $\mathcal N_{\mathcal M}^\delta(z)$ coincides with $\mathcal N_{\mathcal M}(\widetilde{\mathcal T}^z_c)$, and otherwise differs from it by at most one unit. The proof of Proposition~\ref{prop.main} relies on the next two lemmas. 
\begin{lem}\label{lem:main1}
Let $\lambda>0$ and $\delta\in (0,1)$ be given. There exists $R>0$, such that for any $n\ge 1$, any $x\in \mathcal L_n$, any $\mathcal M\in \Lambda_{n,x}^F$, any $L_0\ge R$, and any $i\ge 0$, 
$$\sup_{z\in \mathcal S_i} \mathbb E\Big[e^{\lambda \cdot \mathcal N_{\mathcal M}^\delta(z)}\Big] \le 1+ \frac{1}{L_i^2 \cdot 3^i\cdot \sqrt{\log L_0}}.$$
Furthermore, the same result holds if in the definition of $\mathcal N_{\mathcal M}^\delta(\cdot)$ we replace $\widetilde {\mathcal T}_c$ by $\mathcal T_c$. 
\end{lem}
Given a simple random walk $(S_k)_{k\ge 0}$ on $\mathbb Z^d$, for $i\ge 0$, denote by $\tau_i$ its total time spent in $\mathcal S_i$: 
$$\tau_i = \sum_{k\ge 0} \mathbf 1\{S_k \in \mathcal S_i\}. $$ 
Our second lemma is the following (recall that we assume $d\ge 5$ in the whole section, but the next lemma actually holds in any dimension $d\ge 3$). 
\begin{lem}\label{lem:main2}
There exists $\varepsilon>0$, such that for any $L_0\ge 1$, 
$$\sup_{z\in \mathbb Z^d} \mathbb E_z\Big[\exp\big(\varepsilon \sum_{i\ge 0} \frac{\tau_i}{L_i^2\cdot 3^i}\big)\Big] <\infty. $$ 
\end{lem}
Assuming these two lemma, one can now give the proof of our main proposition. 
\begin{proof}[Proof of Proposition~\ref{prop.main}]
Let $n\ge 1$, $x\in \mathcal L_n$, and $\mathcal M\in \Lambda_{n,x}^F$ be given.  
Recall that we consider here a tree-indexed random walk starting from some $z\in \mathcal X_{\mathcal M}^F$, and one has to bound some exponential moment of the number of frames entering the definition of $\mathcal X_{\mathcal M}^F$, which are visited by this walk, on the event that it avoids $\mathcal X_{\mathcal M}^F$ in the past. A first observation is that in dimension $d\ge 6$, one could just ignore this last event, since (at least on average) its probability is of order one. However, the case of dimension five requires more care, since in this case its probability is (on average) of order $1/\log (L_0)$, as shown by~\eqref{BCapSLF}. We address this issue by decomposing the range of the walk in two parts, corresponding to its trajectory that happens respectively before and after some time $\sigma_0$ defined below. To be more precise, denote by $\Gamma=(\Gamma(k))_{k\ge 0}$ the trajectory of the walk on the spine of $\mathcal T$ (which by definition has the law of a simple random walk starting from $z$), and let 
$$\sigma_0 = \inf\{k: \Gamma(k) \notin \mathcal S_0\}.$$
Let $\mathcal T^z_-(\sigma_0)$ and $\mathcal T^z_+(\sigma_0)$ be the ranges of the walk on the set of trees in the past, respectively the future, attached to the vertices of the spine $\Gamma(k)$, with $k\le \sigma_0$. Ignoring the event that after time $\sigma_0$, the walk in the past has to avoid $\mathcal X_{\mathcal M}^F$, and using that the range after time $\sigma_0$ is independent of its range before, conditionally on its position at that time, we can bound for any $z\in \mathcal X_{\mathcal M}^F$, 
\begin{equation}\label{eq:Nmz}
\mathbb E\Big[e^{\lambda \cdot \mathcal N_{\mathcal M}(\mathcal T^z)}\cdot \mathbf 1\{ \mathcal T^z_- \cap \mathcal X_{\mathcal M}^F = \emptyset\}\Big] \le \mathbb E\Big[e^{\lambda \cdot \mathcal N_{\mathcal M}(\mathcal T^z_+(\sigma_0))}\cdot \mathbf 1\{ \mathcal T^z_-(\sigma_0) \cap \mathcal X_{\mathcal M}^F = \emptyset\}\Big]\cdot \sup_{y\in \mathbb Z^d} \mathbb E\Big[e^{\lambda \cdot \mathcal N_{\mathcal M}(\mathcal T^y_+)}\Big].
\end{equation}
Now, observe that by Corollary~\ref{cor.hit} and~\eqref{BCapSL0}, one can ensure by taking $L_0$ large enough, that 
$$\sup_{y\notin \mathcal S_0} \mathbb P(\mathcal T^y_-\cap \mathcal X_{\mathcal M}^F\neq \emptyset) \lesssim \sum_{i\ge 0} \frac{2^i\cdot  \textrm{BCap}(S_F(0,L_0-1))}{L_i^{d-2}} \le \frac 12,$$
and as a consequence, that
$$\mathbb P(\mathcal T^z_-(\sigma_0) \cap \mathcal X_{\mathcal M}^F = \emptyset) \le 2 \cdot e_{\mathcal X_{\mathcal M}^F}(z). $$ 
Moreover, letting $N_0$ be the total number of offspring in the future of the vertices on the spine up to time $\sigma_0$, one has for some constant $c>0$, 
$$\mathbb P(N_0\ge L_0^2 \cdot \sqrt{\log L_0})
\le e^{-c\sqrt{\log L_0}},$$
since $N_0$ is a sum of $\sigma_0$ independent random variable with law $\widetilde \mu$, and both $\widetilde \mu$ by hypothesis, and (as is well known) $\sigma_0/L_0^2$ have a finite exponential moment. 
Using these last two facts together with Lemma~\ref{lem:main1} (with $i=0$), and recalling that the trees attached to the spine are adjoint trees (except the one attached to the root, which is a usual $\mu$-Bienaym\'e-Galton-Watson tree), we get that for $L_0$ large enough, 
\begin{align}\label{pastfuture}
\nonumber \mathbb E\Big[e^{\lambda \cdot \mathcal N_{\mathcal M}(\mathcal T^z_+(\sigma_0))}& \cdot \mathbf 1\{ \mathcal T^z_-(\sigma_0) \cap \mathcal X_{\mathcal M}^F = \emptyset\}\Big]
 \le e^\lambda\cdot  \mathbb E\Big[\exp\big(\frac{N_0}{L_0^2\cdot  \sqrt{\log L_0}}\big)\cdot \mathbf 1\{ \mathcal T^z_-(\sigma_0) \cap \mathcal X_{\mathcal M}^F = \emptyset\}\Big]\\
\nonumber & \le 2e^{1+\lambda}\cdot e_{\mathcal X_{\mathcal M}^F}(z) + e^{\lambda} \cdot \mathbb E\Big[\exp\big(\frac{N_0}{L_0^2\cdot  \sqrt{\log L_0}}\big)\cdot \mathbf 1\{ N_0 \ge L_0^2\cdot \sqrt{\log L_0}\}\Big] \\
& \le 2 e^{1+\lambda}\cdot e_{\mathcal X_{\mathcal M}^F}(z) +  e^{-c\cdot \sqrt{\log L_0}},
\end{align}
for some possibly smaller constant $c>0$. Note that by summing these inequalities over $z\in \mathcal X_{\mathcal M}^F$, and dividing by $\textrm{BCap}(\mathcal X_{\mathcal M}^F)$, we get a bounded expression, thanks to~\eqref{BCapSLF}.

Hence it just remains to see that the last term in~\eqref{eq:Nmz} is finite. Fix $y\in \mathbb Z^d$, and denote by $\Gamma^y$ the range of the walk on the spine of $\mathcal T^y$. 
For $A\subset \mathbb Z^d$, and $\delta>0$, let 
$$\mathcal N_{\mathcal M}^{\delta,*}(A) = \sum_{m\in T_{(n)}} \mathbf 1\big\{d(A,\square_m)\le \delta L_0\big\}. $$

One has by definition, for any $\delta\in (0,1)$, 
\begin{equation}\label{ineq:NMdelta}
\mathcal N_{\mathcal M}(\mathcal T^y_+) \le \mathcal N_{\mathcal M}^{\delta,*}(\Gamma^y)+ \sum_{k\ge 0} \mathcal N_{\mathcal M}^\delta(\Gamma(k)),
\end{equation}
where with a slight abuse of notation we assume that, conditionally on $\Gamma$, the different terms of the sum above are independent random variables. Now, we first handle the term $\mathcal N_{\mathcal M}^{\delta,*}(\Gamma^y)$, which counts the number of frames composing  $\mathcal X_{\mathcal M}^F$ to which the random walk $\Gamma$ enters the $(\delta L_0)$-neighborhood. As in~\cite{R}, we will dominate the law of this random variable by a Geometric distribution with sufficiently small expectation by choosing $\delta$ small enough. First note that the $(\delta L_0)$-neighborhood of any frame $\square_m$ can be covered by order $1/\delta$ balls of radius $\delta L_0$, and as such has a Newtonian capacity bounded by a constant times $\delta^{d-3} L_0^{d-2}$. It follows that starting from any point on the boundary of one of these neighborhoods, the probability that the random walk ever hits another one is bounded by 
$$C \sum_{i\ge 0} \frac{2^i \cdot \delta^{d-3} L_0^{d-2}}{L_i^{d-2}},$$
for some constant $C>0$, since for any $i\ge 0$ and any given $m\in T_{(n)}$, the number of vertices $m'\in T_{(n)}$, such that $\mathcal M(m')$ is at distance smaller than $L_i$ from $\mathcal M(m)$ is of order $2^i$ by construction. The latter sum can be bounded by $e^{-4\lambda}$, by choosing $\delta$ small enough (independently of the value of $L_0$). 
For such $\delta$, the random variable $\mathcal N_{\mathcal M}^{\delta,*}(\Gamma^y)$ is stochastically dominated by a Geometric random variable with mean $(1-e^{-4\lambda})^{-1}$, and hence we get 
\begin{equation}\label{NMdelta*}
\sup_{y\in \mathbb Z^d} \mathbb E\big[e^{2\lambda \cdot \mathcal N_{\mathcal M}^{\delta,*}(\Gamma^y)}\big]<\infty.
\end{equation}
We next handle the other terms in~\eqref{ineq:NMdelta}, with $\delta$ now fixed as above. Conditioning first on $\Gamma$ and  applying Lemma~\ref{lem:main1}, together with   Lemma~\ref{lem:main2} afterwards, shows that for $L_0$ large enough, one has 
$$\sup_{y\in \mathbb Z^d} \mathbb E\big[e^{2\lambda \cdot \sum_{k\ge 0} \mathcal N_{\mathcal M}^\delta(\Gamma(k))}\big] <\infty. $$ 
Combining the last two displays and~\eqref{ineq:NMdelta}, shows using Cauchy-Schwarz inequality that 
the last term in~\eqref{eq:Nmz} is bounded, as wanted. Altogether this concludes the proof of Proposition~\ref{prop.main}. 
\end{proof}

We now need to prove the two remaining Lemma~\ref{lem:main1} and~\ref{lem:main2}. First we state and prove an intermediate result, that will be needed for the proof of Lemma~\ref{lem:main1}. 

\begin{lem} \label{lem:SRW} Assume $d\ge 3$. 
For $z\in \mathbb Z^d$, let $H_z$ be the hitting time of a point $z$ by a simple random walk, and for $R>0$, let $H_R$ be the hitting time of $\partial B(0,R)$. Fix $K\ge 2$. There exist positive constants $\varepsilon$ and $C$, such that for any 
$R\ge 1$, and any $z\in \partial B(0,R)$, 
$$\mathbb E\big[\exp\big(\frac{\varepsilon \cdot \tau_z}{R^2}\big)\cdot \mathbf 1\{H_z < H_{KR}\} \big]\le C \cdot R^{2-d}. $$ 
\end{lem}
\begin{proof}
We let $(S_k)_{k\ge 0}$ be a simple random walk and write 
\begin{align*}
\mathbb E\big[\exp\big(\frac{\varepsilon \cdot \tau_z}{R^2}\big)\cdot \mathbf 1\{H_z < H_{KR}\} \big] & \le \sum_{k\ge 0} e^{\varepsilon (k+1)}\cdot  \mathbb P\big(H_z <H_{KR}, kR^2\le H_z < (k+1)R^2\big) \\ 
& \le \sum_{k\ge 0} e^{\varepsilon (k+1)}\cdot \sum_{t=kR^2}^{(k+1)R^2} \mathbb P(\max_{m\le t/2} \|S_m\|\le KR, S_t= z). 
\end{align*}
Then to compute the last probability, we use the Markov property at time $\lfloor t/2\rfloor$, and the facts that on one hand, $$\mathbb P(\max_{m\le \lfloor t/2\rfloor} \|S_m\|\le K R) \le \exp(-ct/R^2),$$
for some $c>0$ (depending on $K$), and on the other hand, uniformly over $y,z\in \mathbb Z^d$, and $t\ge 1$, one has for some constant $C>0$, 
$$\mathbb P_y(S_{\lceil t/2\rceil} =z) \le C\cdot t^{-d/2}.$$  
Altogether, this gives for $\varepsilon<c$,
$$\mathbb E\big[\exp\big(\frac{\varepsilon \cdot \tau_z}{R^2}\big)\cdot \mathbf 1\{H_z < H_{KR}\} \big]\lesssim R^{2-d} \sum_{k\ge 0} e^{\varepsilon (k+1)-ck}\lesssim R^{2-d},$$ 
as wanted. 
\end{proof}

\begin{proof}[Proof of Lemma~\ref{lem:main1}]
First of all, observe that it suffices to prove the result for a $\mu$-Bienaym\'e-Galton-Watson tree in the definition of $\mathcal N_{\mathcal M}^\delta$, instead of an adjoint tree, say with a factor $(\log L_0)^{3/4}$ in the upper bound instead of $\sqrt{\log L_0}$.  Indeed, the result for the latter follows from the former one, just by conditioning on the number of offspring of the root, and using that $\mu$ has some finite exponential moment.  So, with a slight abuse of notation, we assume now that in the definition of  $\mathcal N_{\mathcal M}^\delta$, the range of the random walk indexed by $\widetilde {\mathcal T_c}$ is replaced by the range of a critical branching random walk.

The overall strategy is to use an induction argument. Namely, we prove the result with $\mathcal N_{\mathcal M}^\delta(\cdot)\wedge k$, in place of $\mathcal N_{\mathcal M}^\delta(\cdot)$, by induction on $k\ge 1$ (where we use the notation $a\wedge b = \min(a,b)$). When $k=1$, we have for any $z\in \mathbb Z^d$, 
$$\mathbb E\Big[e^{\lambda \cdot \mathcal N_{\mathcal M}^\delta(z)\wedge 1}\Big] \le 1 + e^{\lambda} \cdot \mathbb P(\mathcal N_{\mathcal M}^\delta(z) \ge 1). 
$$ 
Thanks to Corollary~\ref{cor.hit}  and~\eqref{BCapSL0}, one has for any $i\ge 0$ and $z\in \mathcal S_i$, for some constant $C>0$ (that might depend on $\delta$),  
$$\mathbb P(\mathcal N_{\mathcal M}^\delta(z) \ge 1) \le C \sum_{j\ge i} \frac{2^j\cdot  \textrm{BCap}(S_F(0,L_0-1))}{L_j^{d-2}}\le \frac{C}{L_i^2\cdot 3^i\cdot \log(L_0)}.$$
We then take $L_0$ large enough, so that $(\log L_0)^{1/4}\ge Ce^{\lambda}$, and this gives the desired statement for $k=1$. We next prove the induction step. So given some $k\ge 1$, we assume that the statement holds for $\mathcal N_{\mathcal M}^\delta(\cdot )\wedge k$, and we prove it now for $\mathcal N_{\mathcal M}^\delta(\cdot)\wedge (k+1)$. One has for any $z\in \mathbb Z^d$, 
$$\mathbb E\Big[e^{\lambda \cdot \mathcal N_{\mathcal M}^\delta(z)\wedge (k+1)}\Big] \le 1 + e^{\lambda}\cdot \mathbb E\Big[e^{\lambda \cdot \mathcal N_{\mathcal M}^\delta(z)\wedge k}\cdot \mathbf 1\{\mathcal N_{\mathcal M}^\delta(z)\ge 1\}\Big]. $$ 
We next use that the law of a critical branching random walk, conditioned on hitting $\mathcal X_{\mathcal M}^F$ can be described explicitly in terms of the path of first visit, thanks to Proposition~\ref{Zhu.hit2}, and the remark following it. Given a path $\gamma:\{0,\dots,N\}\to \mathbb Z^d$, we let $\mathcal R(\gamma) = \{\gamma(\ell):0\le \ell \le N\}$ be its range, and $s(\gamma)$ the probability that a simple random walk follows the path $\gamma$. We also let $\mathcal T_-^\gamma$ be the union of the ranges of a sequence of independent random walks indexed by adjoint trees, starting from the points $\gamma(\ell)$, for each $\ell\ge 1$. A similar inequality as~\eqref{ineq:NMdelta} for $\mathcal N_{\mathcal M}^\delta(z)$ yields
\begin{align}\label{sumgamma}
\nonumber  \mathbb E\Big[e^{\lambda \cdot \mathcal N_{\mathcal M}^\delta(z)\wedge (k+1)}\Big] 
&  \le 1+e^\lambda  \sum_{\substack{m\in T_{(n)} : \\ d(z,\square_m)\ge \delta L_0}} \sum_{y\in \square_m} \sum_{\gamma : y\to z} s(\gamma)  \cdot \exp\big(\lambda\cdot  \mathcal N_{\mathcal M}^{\delta,*}(\mathcal R(\gamma))\big)\\
 & \quad \times \mathbb E\Big[\exp\Big(\lambda\sum_{\ell} \mathcal N_{\mathcal M}^\delta(\gamma(\ell))\wedge k\Big)\cdot \mathbf 1\{\mathcal T_-^\gamma\cap \mathcal X_{\mathcal M}^F = \emptyset\}\Big], 
\end{align}
where the third sum on the right hand side runs over paths $\gamma$ that go from $y$ to $z$ and avoid $\mathcal X_{\mathcal M}^F$, except at their starting point.

Set for $m\in T_{(n)}$
$$U_m^\delta = \{u\in \mathbb Z^d : d(u,\square_m)\ge \delta L_0/2\}.$$
For any $m\in T_{(n)}$, one has by using the same argument as in~\eqref{pastfuture} together with the induction hypothesis, for some constant $C>0$, 
$$\sum_{y\in \square_m} 
\sum_{\substack{\gamma : y \to \partial U_m^\delta \\ \gamma \subset (U_m^\delta)^c}} 
s(\gamma)\cdot  \mathbb E\Big[\exp\Big(\lambda\sum_{\ell} \mathcal N_{\mathcal M}^\delta(\gamma(\ell))\wedge k\Big)\cdot \mathbf 1\{\mathcal T_-^\gamma\cap \mathcal X_{\mathcal M}^F = \emptyset\}\Big]\le C\cdot \textrm{BCap}(\square_m), $$ 
where the second sum runs over paths $\gamma$ that remain in $(U_m^\delta)^c$ except at their last step where they reach $\partial U_m^\delta$. Using this in~\eqref{sumgamma} for the part of each path $\gamma$ until its first hitting time of $U_m^\delta$, forgetting about 
the indicator function appearing there for the remaining part of $\gamma$, and using the induction hypothesis, gives
\begin{align}\label{sumgamma2}
\nonumber  \mathbb E\Big[e^{\lambda \cdot \mathcal N_{\mathcal M}^\delta(z)\wedge (k+1)}\Big] 
 & \le 1+ Ce^\lambda   \sum_{\substack{m\in T_{(n)} : \\ d(z,\square_m)\ge \delta L_0}} \textrm{BCap}(\square_m)\cdot\sup_{w\in \partial U_m^\delta} \sum_{\gamma : w\to z} s(\gamma)  \cdot \exp\big(\lambda\cdot  \mathcal N_{\mathcal M}^{\delta,*}(\mathcal R(\gamma))\big)\\
 & \qquad \times 
 \exp\Big(\sum_{j\ge 0} \frac{\tau_j(\gamma)}{L_j^2 \cdot 3^j \cdot (\log L_0)^{3/4}} \Big), 
\end{align}
where for $j\ge 0$, we let $\tau_j(\gamma) = \sum_\ell \mathbf 1\{\gamma(\ell)\in\mathcal S_j\}$ be the time spent in $\mathcal S_j$ by $\gamma$.

Assume now that $z\in \mathcal S_i$, for some $i\ge 0$. We first treat the sum above over points $m\in T_{(n)}$, such that $d(z,\square_m)\le L_{i+2}$. Given such $m\in T_{(n)}$, $w\in \partial U_m^\delta$, and $\gamma$ some path from $w$ to $z$, we again cut it into two parts, one before its hitting time of $z$, and the other one after that time. Concerning the second part, note that 
\begin{align*}
& \sum_{\gamma : z \to z} s(\gamma)  \cdot \exp\Big(\lambda\cdot  \mathcal N_{\mathcal M}^{\delta,*}(\mathcal R(\gamma)) 
+\sum_{j\ge 0} \frac{\tau_j(\gamma)}{L_j^2 \cdot 3^j \cdot (\log L_0)^{3/4}} \Big) \\
& \le \mathbb E_z\Big[
 \exp\Big(\lambda\cdot  \mathcal N_{\mathcal M}^{\delta,*}(\mathcal R_\infty)
+\sum_{j\ge 0} \frac{\tau_j}{L_j^2 \cdot 3^j \cdot (\log L_0)^{3/4}}\Big)\cdot \mathcal L_\infty(z)\Big],
 \end{align*}
 where $\mathcal L_\infty(z)$ denotes the total time spent in $z$ by the simple random walk, $\mathcal R_\infty$ its range, and $\tau_j$ its time spent in $\mathcal S_j$. Hence using Cauchy-Schwarz inequality, together with Lemma~\ref{lem:main2} and~\eqref{NMdelta*}, we can see that this sum above is finite, bounded by a constant independent of $z$, provided $L_0$ is large enough. It remains now to handle the first part of paths $\gamma$ in~\eqref{sumgamma2}, up to their hitting time of $z$. Given some $\gamma$, let $\sigma_i(\gamma)$ be the last time $\gamma$ is in $\mathcal S_{i-3}\cup \mathcal S_{i+3}$ before hitting $z$, with the convention that $\mathcal S_j=\emptyset$, if $j<0$. Let also $\sigma'_i(\gamma)$ be either $0$, if $\sigma_i(\gamma) = 0$ (which can only happen if $i\le 2$ and $\gamma$ never enters $\mathcal S_{i+3}$ before hitting $z$), or the first time it enters $\mathcal S_{i-1}\cup \mathcal S_{i+1}$, if $\sigma_i(\gamma)>0$. The part of $\gamma$ after $\sigma'_i(\gamma)$ is handled using Lemma~\ref{lem:SRW}, which shows that if  $L_0$ is large enough, then in case when  $\sigma'_i(\gamma)>0$, for some constant $C>0$, 
 $$\sup_{u\in \partial_-S_{i-1}\cup \partial_+\mathcal S_{i+1}} \sum_{\substack{\gamma: u\to z\\ \gamma \subset \bigcup_{i-2\le j\le i+2} \mathcal S_j}} s(\gamma)\cdot \exp\Big(\sum_{j=i-2}^{i+2} \frac{\tau_j(\gamma)}{L_j^2 \cdot 3^j \cdot (\log L_0)^{3/4}} \Big) \le C L_i^{2-d}, $$
where the sum runs over paths that hit $z$ at their last step and remain in $\bigcup_{i-2\le j\le i+2} \mathcal S_j$, and where $\partial_-\mathcal S_{i-1}$ denotes the boundary of $\mathcal S_{i-1}$ that lies at the interface with $\mathcal S_{i-2}$ and similarly $\partial_+ \mathcal S_{i+1}$ is the boundary of $\mathcal S_{i+1}$ at the interface with $\mathcal S_{i+2}$. The case when $\sigma'_i(\gamma)=0$ is handled as well using Lemma~\ref{lem:SRW} (note furthermore that in this case $\mathcal N_{\mathcal M}^{\delta,*}(\mathcal R(\gamma))$ is bounded by a deterministic constant).

Concerning the part of $\gamma$ before time $\sigma'_i(\gamma)$ (in case it is positive), consider first the case when $\gamma(\sigma'_i(\gamma))\in \partial \mathcal S_{i-1}$. Then by concatenating $\gamma$ with an infinite nearest neighbor path that never hit $\mathcal S_{i-3}$ and diverges to infinity, we obtain a new infinite random walk trajectory, for which $\sigma'_i(\gamma)$ is the first hitting time of $\mathcal S_{i-1}$ after the last visit to $\mathcal S_{i-3}$. It follows that the corresponding sum over such paths is bounded by 
$$\frac 1{\inf_{u\in \partial_- S_{i-1}} \mathbb P_u(H_{\mathcal S_{i-3}}=\infty)} \mathbb  E_w\Big[
 \exp\Big(\lambda\cdot  \mathcal N_{\mathcal M}^{\delta,*}(\mathcal R_\infty)
+\sum_{j\ge 0} \frac{\tau_j}{L_j^2 \cdot 3^j \cdot (\log L_0)^{3/4}}\Big)\Big], 
$$
which we already saw is bounded by a constant, for $L_0$ sufficiently large. Consider next the case when $\gamma(\sigma'_i(\gamma))\in \partial \mathcal S_{i+1}$. By concatenating $\gamma$ with an infinite nearest neighbor path that diverges to infinity and never hit $\mathcal S_i$, we end up with a path for which $\sigma'_i(\gamma)$ becomes the first hitting time of $\mathcal S_{i+1}$, after the first hitting time of $\mathcal S_{i+3}$ after the last visit to $\mathcal S_i$. Hence, similarly as in the previous case, the corresponding sum is bounded by  
$$\frac 1{\inf_{u\in \partial_+ S_{i+1}} \mathbb P_u(H_{\mathcal S_i}=\infty)} \mathbb  E_w\Big[
 \exp\Big(\lambda\cdot  \mathcal N_{\mathcal M}^{\delta,*}(\mathcal R_\infty)
+\sum_{j\ge 0} \frac{\tau_j}{L_j^2 \cdot 3^j \cdot (\log L_0)^{3/4}}\Big)\Big], 
$$
which is bounded as well for $L_0$ large enough. Altogether, this shows that if $L_0$ is large enough, then for some $C>0$,  
\begin{equation}\label{borneLi}
  \sum_{\substack{m\in T_{(n)} : \\ \delta L_0 \le d(z,\square_m)\le L_{i+3}}} \sup_{w\in \partial U_m^\delta} \sum_{\gamma : w\to z} s(\gamma)  \cdot \exp\Big(\lambda\cdot  \mathcal N_{\mathcal M}^{\delta,*}(\mathcal R(\gamma))
+\sum_{j\ge 0} \frac{\tau_j(\gamma)}{L_j^2 \cdot 3^j \cdot (\log L_0)^{3/4}} \Big)\le C\cdot \frac{2^i}{L_i^{d-2}}.  
\end{equation} 
Finally, given some $\ell\ge i+2$, we treat the sum over elements $m\in T_{(n)}$, such that $L_\ell\le d(z,\square_m)\le L_{\ell+1}$. For such $m$, some $w\in \partial U_m^\delta$, and some path $\gamma$ from $w$ to $z$, we consider  $t_\ell(\gamma)$, the last time $\gamma$ visits $\mathcal S_{\ell-1}$. Looking at the time inverse $\stackrel{\leftarrow}{\gamma}$ of $\gamma$ (meaning the same path but traveling backward from $z$ to $w$), $t_\ell(\gamma)$ becomes the first time $\stackrel{\leftarrow}{\gamma}$ hits $\mathcal S_{\ell - 1}$. Then we have to control the sum over paths from $w$ up to some
fixed point $z'\in \partial S_{\ell -1}$, which can be done exactly using the previous argument, and paths from $z$ up to their hitting time of $\mathcal S_{\ell-1}$, which  is bounded by 
$$\mathbb  E_z\Big[
 \exp\Big(\lambda\cdot  \mathcal N_{\mathcal M}^{\delta,*}(\mathcal R_\infty)
+\sum_{j\ge 0} \frac{\tau_j}{L_j^2 \cdot 3^j \cdot (\log L_0)^{3/4}}\Big)\Big].$$
Altogether this shows that for any $\ell\ge i+2$, 
\begin{equation}\label{borneLj}
  \sum_{\substack{m\in T_{(n)} : \\ L_\ell \le d(z,\square_m)\le L_{\ell+1}}} \sup_{w\in \partial U_m^\delta} \sum_{\gamma : w\to z} s(\gamma)  \cdot \exp\Big(\lambda\cdot  \mathcal N_{\mathcal M}^{\delta,*}(\mathcal R(\gamma))
+\sum_{j\ge 0} \frac{\tau_j(\gamma)}{L_j^2 \cdot 3^j \cdot (\log L_0)^{3/4}} \Big)\le C\cdot \frac{2^\ell}{L_\ell^{d-2}}.  
\end{equation}
Combining~\eqref{BCapSL0},~\eqref{sumgamma2},~\eqref{borneLi},  and~\eqref{borneLj}, and taking again larger $L_0$ if necessary, concludes the proof of the induction step and of the lemma.  
\end{proof}

\begin{proof}[Proof of Lemma~\ref{lem:main2}]
We show that for $\varepsilon>0$ small enough,
$$\sup_{z\in \mathbb Z^d} \sup_{I\ge 1} \mathbb E_z\Big[\exp\big(\varepsilon
\sum_{i= 0}^I \frac{\tau_i}{L_i^2\cdot 3^i}\big)\Big] <\infty,$$
which will prove the result by monotone convergence. Using successively Cauchy-Schwarz inequality, and then Jensen's inequality gives 
$$\mathbb E_z\Big[\exp\big(\varepsilon
\sum_{i= 0}^I \frac{\tau_i}{L_i^2\cdot 3^i}\big)\Big] \le \prod_{i=0}^I \mathbb E_z\Big[\exp(\varepsilon\cdot \frac{\tau_i \cdot 2^{i+1}}{L_i^2 \cdot 3^i}\big)\Big]^{1/2^{i+1}}\le \prod_{i=0}^I \mathbb E_z\Big[\exp(\varepsilon\cdot \frac{\tau_i}{L_i^2}\big)\Big]^{1/3^i}. $$ 
Hence it suffices to show that for $\varepsilon>0$ small enough, 
$$\sup_{z\in \mathbb Z^d} \sup_{i\ge 0} \mathbb E_z\Big[\exp(\varepsilon\cdot \frac{\tau_i}{L_i^2}\big)\Big]<\infty. $$
In fact the result holds even with $\widetilde \tau_i = \tau_0+\dots+\tau_i$ instead of $\tau_i$ in the exponential. Indeed, note that for any $i\ge 0$, $\widetilde \tau_i$ is a sum of the times spent in $\cup_{j\le i}\mathcal S_j$, during excursions from the boundary of $\cup_{j\le i} \mathcal S_j$, up to the boundary of $\mathcal S_{i+K}$, with $K$ some fixed constant taken large enough later. During each of these excursions, it is well known that the time spent in $\cup_{j\le i} \mathcal S_j$ divided by $L_i^2$ has some finite exponential moment, uniformly in the starting point and in $i\ge 0$ (but with $K$ fixed). So in fact we just need to show that the number of excursions has a finite exponential moment. However, a union bound over all boxes forming $\cup_{j\le i} \mathcal S_j$, shows that starting from any point on $\partial \mathcal S_{i+K}$, the probability to ever hit one of those is uniformly bounded by a constant smaller than one, at least if $K$ is taken large enough. Hence the number of excursions is stochastically bounded by a Geometric random variable, and as such it has indeed some finite exponential moment. This concludes the proof of the lemmma. 
\end{proof}
\section{Proof of Theorem~\ref{thm.cover}}
\label{sec:thm.cover}
We follow carefully the proof of Belius~\cite{Bel12}. For $x\in \mathbb Z^d$, we let $U_x =\inf\{u\ge 0 : x\in \mathcal I^u\}$, so that for $K\subset \mathbb Z^d$, $M(K) = \max \{U_x : x\in K\}$. 
The first step is the following lemma. 

\begin{lem}\label{lem.Gumbel}
Let $\lambda>0$. There exists $c>0$, such that the following holds. Let $K\subset \mathbb Z^d$, finite and nonempty, for which $\|x-y\|\ge |K|^{(2+\lambda)/(d-4)}$, for all distinct $x,y\in K$. Then for any $u\ge 0$, 
$$|\mathbb P(M(K)\le u) - \mathbb P(U_0\le u)^{|K|}| \le c u |K|^{-\lambda}. $$  
\end{lem}
\begin{proof}
Using theorem~\ref{thm.cov}, we get that for any $x\in K$, 
$$|\mathbb P(M(K)\le u) - \mathbb P(M(K\setminus \{x\}) \le u)\cdot \mathbb P(U_x\le u)| \lesssim  u |K|^{-1-\lambda}. $$ 
The proof follows by using an induction on the cardinality of $K$ (and the fact that $U_x$ and $U_0$ have the same law, for all $x$). 
\end{proof}

Now for $K\subset \mathbb Z^d$ finite and nonempty, and $\varepsilon>0$, let 
$$K_\varepsilon = \Big\{x\in K : U_x\ge \frac{(1-\varepsilon)}{\rm{BCap}(\{0\})}\log |K|\Big\}.$$
\begin{lem}\label{lem.sumKeps}
There exists $\varepsilon_0>0$, such that for any $\varepsilon\in (0,\varepsilon_0)$ the following holds. There exists $c>0$, such that for any $K\subset \mathbb Z^d$ finite and nonempty, 
\begin{equation}\label{Keps1}
\sum_{x\neq y} \mathbb P(x,y\in K_\varepsilon) \le 
|K|^{2\varepsilon} + c|K|^{-\varepsilon},
\end{equation}
and for any $\eta\in (3\varepsilon,1)$, 
\begin{equation}\label{Keps2}
\sum_{0<\|x-y\|<|K|^{(1-\eta)/d}} \mathbb P(x,y\in K_\varepsilon) \le c\,  |K|^{- \varepsilon}. 
\end{equation}
\end{lem}
\begin{proof}
It suffices to prove that for any $a>0$, 
\begin{equation}\label{Keps3}
\sum_{0<\|x-y\|<a }\mathbb P(x,y\in K_\varepsilon) \le \min\{|K|^{2\varepsilon}, c|K|^{2\varepsilon-1}a^d\} + c|K|^{-\varepsilon}.
\end{equation}
Indeed, once this is proved, then~\eqref{Keps2} and~\eqref{Keps1}  follow respectively by taking $a=|K|^{(1-\eta)/d}$, and letting $a\to \infty$. We split the sum in~\eqref{Keps3} in two parts: 
$$ I_1 = \sum_{0<\|x-y\|<(\log |K|)^2}\mathbb P(x,y\in K_\varepsilon)\qquad {\rm and}\qquad I_2 = \sum_{(\log |K|)^2\le \|x-y\|<a }\mathbb P(x,y\in K_\varepsilon). $$ 
Using Lemma~\ref{strict.bcap}, we get with $\delta$ from therein and $\varepsilon$ small enough,  
$$I_1 \lesssim (\log |K|)^{2d}\cdot |K|^{1 - (1-\varepsilon) \frac{{\rm BCap}(\{x,y\})}{{\rm BCap}(\{0\})}} \lesssim (\log |K|)^{2d}\cdot |K|^{1 - (1-\varepsilon)(1+\delta)}\lesssim |K|^{-\varepsilon}.  $$
Using now Lemma~\ref{lem.Bcap0x}, one can see that for $\varepsilon$ small enough, and some $c>0$,  
\begin{align*}
I_2 & \le |K|^{-2(1-\varepsilon)} \sum_{x,y\in K:(\log |K|)^2\le \|x-y\|<a} (1+ c\frac{\log |K|}{\|x-y\|^{d-4}}) \\
& \le \min\{|K|^{2\varepsilon}, c|K|^{2\varepsilon-1}a^d\} + c(\log |K|)\cdot  |K|^{2\varepsilon-1+4/d}\\
& \le \min\{|K|^{2\varepsilon}, c|K|^{2\varepsilon-1}a^d\} + c|K|^{-\varepsilon},
\end{align*}
concluding the proof of~\eqref{Keps3}. 
\end{proof}
Now for $K\subset \mathbb Z^d$, finite and nonempty and $\varepsilon>0$, we define a set of good subsets of $K$ as  
$$G_{K,\varepsilon} = \Big\{U\subset K : \big||U| - |K|^\varepsilon\big| \le |K|^{2\varepsilon/3}, U\neq \emptyset, \text{ and } \|x-y\|>|K|^{\frac{4\varepsilon}{d-4}} \ \forall x\neq y\in U\Big\}. $$ 
\begin{lem}\label{lem.goodKeps}
There exists $\varepsilon_0>0$, such that for any $K\subset \mathbb Z^d$, finite and nonempty, and any $\varepsilon \in (0,\varepsilon_0)$, one has for some $c>0$, 
$$\mathbb P(K_\varepsilon \notin G_{K,\varepsilon}) \le c|K|^{-\varepsilon/3}. $$
\end{lem}
\begin{proof}
Note that 
$$\mathbb E[|K_\varepsilon|] = |K|\cdot \mathbb P\Big(U_0\ge \frac{(1-\varepsilon)}{{\rm BCap}(\{0\})} \log |K|\Big) =|K|^\varepsilon.$$ 
Thus, it follows from Chebyshev's inequality and~\eqref{Keps1} that 
$$\mathbb P\big(\big||K_\varepsilon| - |K|^\varepsilon\big|>|K|^{2\varepsilon/3}\big) \lesssim |K|^{-\varepsilon/3}. $$ 
Furthermore, a union bound and~\eqref{Keps2} yield
$$\mathbb P\big(\exists x\neq y\in K_\varepsilon: \|x-y\| \le |K|^{\frac{4\varepsilon}{d-4}}
\big) \lesssim |K|^{-\varepsilon},$$
provided $\varepsilon$ is small enough. 
\end{proof}
One can now conclude the proof. 
\begin{proof}[Proof of Theorem~\ref{thm.cover}]
Let $u_K(z) =  \frac{z+ \log |K|}{{\rm BCap}(\{0\})}$. 
If $z\le - (\varepsilon/2) \log |K|$, then by Lemma~\ref{lem.goodKeps}, 
$$\mathbb P(M(K) \le u_K(z) ) \le \mathbb P(K_{\varepsilon/2} = \emptyset) \le \mathbb P(K_{\varepsilon/2}\notin G_{K,\varepsilon/2}) \lesssim K^{-\varepsilon/2},$$
and $$\exp(-e^{-z}) \le \exp(- |K|^{-\varepsilon/2})\lesssim |K|^{-\varepsilon}.$$
Likewise, if $z\ge \log |K|$, then 
$$\mathbb P(M(K)> u_K(z)) \lesssim |K|^{-1},$$
and 
$$\exp(-e^{-z})\ge \exp (-|K|^{-1}) \ge 1 - |K|^{-1}. $$  
Hence it just remains to consider the case $z\in [-(\varepsilon/2)\log |K|,\log |K|]$. By Lemma~\ref{lem.goodKeps}, one may assume that $K_\varepsilon \in G_{K,\varepsilon}$, and condition on the event $\{K_\varepsilon = U\}$, for some $U\in G_{K,\varepsilon}$. Note also that the Branching Interlacements point process between levels $u_1:=\tfrac{(1-\varepsilon)}{{\rm BCap}(\{0\})} \log |K|$ and $u_K(z)$ is independent of $\mathcal I^{u_1}$ and distributed as $\mathcal I^{u_2}$, where $u_2 = u_K(z) - u_1 = \tfrac{z+\varepsilon \log |K|}{{\rm BCap}(\{0\})}$. Therefore, it just amounts to show that for any fixed $U\in G_{K,\varepsilon}$,
\begin{equation*}\label{goal.thmcover}
|\mathbb P(M(U) \le u_2)- \exp(-e^{-z})|\lesssim |K|^{-\varepsilon},
\end{equation*}
which follows directly from Lemma~\ref{lem.Gumbel}, and the fact that  $|(1-\tfrac{v}{n})^n - e^{-v}|\lesssim 1/n$, uniformly in $v\in [0,1]$. 
\end{proof}

\textbf{Acknowledgments:} Many thanks to Amine Asselah for our numerous discussions around this project, which in particular led to the proofs of Theorems~\ref{thm.cov} and~\ref{thm.cover}. I also thank Bal\'azs R\'ath for comments on a first version of this paper.


\begin{thebibliography}{99}

\bibitem{ARZ} O. Angel, B. R\'ath, Q. Zhu. Branching Interlacements, unpublished. 

\bibitem{ASS23} A. Asselah, B. Schapira. P. Sousi. Local times and capacity for transient branching random walks, (2023), arXiv:2303.17572. 

\bibitem{AN} K. B. Athreya, P. E. Ney. Branching processes. Reprint of the 1972 original. Dover Publications, Inc., Mineola, NY, 2004. xii+287 pp.

\bibitem{BW22} T. Bai, Y. Wan. Capacity of the range of tree-indexed random walk. Ann. Appl. Probab. 32 (2022), 1557--1589.


\bibitem{Bel12} D. Belius. Cover levels and random interlacements. Ann. Appl. Probab. 22 (2012), 522--540.

\bibitem{Bel13} D. Belius. Gumbel fluctuations for cover times in the discrete torus. 
Probab. Theory Related Fields 157 (2013), 635--689.

\bibitem{BC} I. Benjamini, N. Curien. Recurrence of the  $\mathbb Z^d$-valued infinite snake via unimodularity. 
Electron. Commun. Probab. 17 (2012), 10 pp.

\bibitem{Bou} N. Bouchot. A confined random walk locally looks like tilted random interlacements, arXiv:2405.14329. 


\bibitem{CTex} J. Cern\'y, A. Teixeira. Random walks on torus and random interlacements: macroscopic coupling and phase transition. Ann. Appl. Probab. 26 (2016), 2883--2914.



\bibitem{DPR} A. Drewitz, A. Pr\'evost, P.-F.  Rodriguez. Cluster capacity functionals and isomorphism theorems for Gaussian free fields. 
Probab. Theory Related Fields 183 (2022),  255--313.

\bibitem{DC1} H. Duminil-Copin, S. Goswami, P.-F. Rodriguez, F. Severo, A. Teixeira. 
 A characterization of strong percolation via disconnection. Proc. Lond. Math. Soc. (3) 129 (2024), Paper No. e12622, 49 pp.

\bibitem{DC2} H. Duminil-Copin, S. Goswami, P.-F. Rodriguez, F. Severo, A. Teixeira. Finite range interlacements and couplings, to appear in Ann. Probab. 

\bibitem{DC3} H. Duminil-Copin, S. Goswami, P.-F. Rodriguez, F. Severo, A. Teixeira. Phase transition for the vacant set of random walk and random interlacements, arXiv:2308.07919

\bibitem{GPF} S. Goswami, P.-F. Rodriguez, Y. Shulzhenko. Strong local uniqueness for the vacant set of random interlacements, arXiv:2503.14497. 

\bibitem{GPF2} S. Goswami, P.-F. Rodriguez, Y. Shulzhenko. Sharp connectivity bounds for the vacant set of random interlacements, arXiv:2504.02777. 


\bibitem{LL} G. F. Lawler; V. Limic. Random walk: a modern introduction. Cambridge Studies in Advanced Mathematics, 123. Cambridge University Press, Cambridge, 2010.

\bibitem{LGL16} J.-F. Le Gall, S. Lin. The range of tree-indexed random walk. 
J. Inst. Math. Jussieu 15 (2016), 271--317.

\bibitem{LSS} A. Legrand, C. Sabot, B. Schapira. Recurrence and transience of the critical random walk snake in random conductances. Electron. J. Probab. 30,  (2025), 1--23. 


\bibitem{Li17} X. Li. A lower bound for disconnection by simple random walk. 
Ann. Probab. 45 (2017), 879--931.

\bibitem{Li20} X. Li. Percolative properties of Brownian interlacements and its vacant set. 
J. Theoret. Probab. 33 (2020), 1855--1893.

\bibitem{LiS} X. Li, A.-S. Sznitman. Large deviations for occupation time profiles of random interlacements. 
Probab. Theory Related Fields 161 (2015),  309--350.


\bibitem{Lupu} T. Lupu. From loop clusters and random interlacements to the free field. Ann. Probab. 44 (2016), 2117--2146.


\bibitem{NSz} M. Nitzschner, A.-S.  Sznitman. Solidification of porous interfaces and disconnection. 
J. Eur. Math. Soc. (JEMS) 22 (2020), 2629--2672.

\bibitem{PTex} S. Popov, A. Teixeira. Soft local times and decoupling of random interlacements. J. Eur. Math. Soc. 17 (2015), 2545--2593.

\bibitem{Pre} A. Pr\'evost.  First passage percolation, local uniqueness for interlacements and capacity of random walk. 
Comm. Math. Phys. 406 (2025), no. 2, Paper No. 34, 75 pp.

\bibitem{PZ} E. B. Procaccia, Y. Zhang. 
Connectivity properties of branching interlacements. 
ALEA Lat. Am. J. Probab. Math. Stat.16 (2019), 279-–314.


\bibitem{R} B. R\'ath. A short proof of the 
phase transition for the vacant set of random interlacements. Electron. Commun. Probab. 20, (2015), 1--11. 




\bibitem{Rod19} P.-F. Rodriguez. On pinned fields, interlacements, and random walk on  $(\mathbb Z/N\mathbb Z)^2$.  
Probab. Theory Related Fields 173 (2019), 1265--1299.


\bibitem{SSz} V. Sidoravicius, A.-S. Sznitman.  Percolation for the vacant set of random interlacements. Comm. Pure Appl. Math. 62 (2009), 831--858.

\bibitem{Sz10} A.-S. Sznitman. Vacant set of random interlacements and percolation. 
Ann. of Math. (2), 171 (2010), 2039--2087.

\bibitem{Sz12} A.-S. Sznitman.  Decoupling inequalities and interlacement percolation on $G\times \mathbb Z$. 
Invent. Math. 187 (2012), 645--706.

\bibitem{Sz12bis} A.-S. Sznitman. Random interlacements and the Gaussian free field. 
Ann. Probab. 40 (2012), 2400--2438.

\bibitem{Sz16} A.-S. Sznitman. Coupling and an application to level-set percolation of the Gaussian free field. 
Electron. J. Probab. 21 (2016), 26 pp.

\bibitem{Sz17} A.-S. Sznitman. Disconnection, random walks, and random interlacements. 
Probab. Theory Related Fields 167 (2017), 1--44.

\bibitem{Sz23} A.-S. Sznitman. On the cost of the bubble set for random interlacements. 
Invent. Math. 233 (2023), 903--950.

\bibitem{Sz23bis} A.-S. Sznitman. On bulk deviations for the local behavior of random interlacements. 
Ann. Sci. \'Ec. Norm. Sup\'er. (4) 56 (2023), 801--858.

\bibitem{Tex} A. Teixeira. On the uniqueness of the infinite cluster of the vacant set of random interlacements. 
Ann. Appl. Probab. 19 (2009), 454--466.


\bibitem{TexW} A. Teixeira, D. Windisch. On the fragmentation of a torus by random walk. Comm. Pure Appl. Math. 64 (2011),  1599--1646.


\bibitem{Zhu16}  Q. Zhu. On the critical branching random walk I: branching capacity and visiting probability,  arXiv:1611.10324. 

\bibitem{Zhu16b}  Q. Zhu. On the critical branching random walk II: branching capacity and branching recurrence, arXiv:1612.00161. 

\bibitem{Zhu18} Q. Zhu. Branching interlacements and tree-indexed random walks in tori. arXiv:1812.10858. 

\bibitem{Zhu19} Q. Zhu. An upper bound for the probability of visiting a distant point by a critical branching random walk in $\mathbb Z^4$. Electron. Commun. Probab. 24 (2019), Paper No. 32, 6 pp.

\bibitem{Zhu21} Q. Zhu. 
On the critical branching random walk III: The critical dimension. Ann. Inst. Henri Poincar\'e Probab. Stat. 57 (2021), 73--93.

\end{thebibliography}
\end{document}